\documentclass[a4paper,12pt]{amsart}

\usepackage[utf8]{inputenc}
\usepackage[T1]{fontenc}
\usepackage[UKenglish]{babel}
\usepackage[a4paper,margin=28mm]{geometry}
\usepackage{verbatim}

\allowdisplaybreaks[4]
\usepackage{times}
\usepackage{dsfont,mathrsfs}
\usepackage{amsmath}
\usepackage{amsthm}
\usepackage{amssymb}
\usepackage{amsfonts}
\usepackage{latexsym}
\usepackage{booktabs}
\usepackage{mathtools}

\usepackage[colorlinks,pagebackref]{hyperref}
\usepackage{xcolor}

\newtheorem{theorem}{Theorem}[section]
\newtheorem{lemma}[theorem]{Lemma}

\newtheorem{assumption}{Assumption}

\theoremstyle{definition}

\numberwithin{equation}{section}

\renewcommand{\labelenumi}{\roman{enumi}}
\renewcommand\theenumi\labelenumi
\renewcommand{\leq}{\leqslant}
\renewcommand{\le}{\leqslant}
\renewcommand{\geq}{\geqslant}
\renewcommand{\ge}{\geqslant}

\newcommand{\R}{\mathbb{R}}
\newcommand{\E}{\mathbb{E}}
\newcommand{\W}{\mathbb{W}}
\newcommand{\LL}{\mathcal{L}}
\newcommand{\PP}{\mathbb{P}}

\newcommand{\dtv}{d_{\mathrm{TV}}}
\newcommand{\dif}{\mathrm{d}}
\newcommand{\eup}{\mathrm{e}}

\newcommand*\abs[1]{\left\lvert#1\right\rvert}
\newcommand*\norm[1]{\left\lVert#1\right\rVert}
\newcommand*\opnorm[1]{\left\lVert#1\right\rVert_{\textup{op}}}
\newcommand*\oinorm[1]{\left\lVert#1\right\rVert_{\textup{op},\infty}}
\newcommand*\sca[1]{\left\langle#1\right\rangle}
\usepackage{marginnote}
\title{tamed Euler-Maruyama Method for SDEs with Non-globally Lipschitz Drift and multiplicative Noise}

\author{Xiang Li \textsuperscript{1,3}}
\author{Yingjun Mo \textsuperscript{1,3}}
\author{Haoran Yang \textsuperscript{2}}

\address{\textsuperscript{1} Department of Mathematics, Faculty of Science and Technology, University of Macau, Macau, 999078, China }

\address{\textsuperscript{2} School of Mathematical Sciences, Peking University, Beijing, 100871, China  }

\address{\textsuperscript{3} Zhuhai UM Science, Technology Research Institute, Zhuhai, 519031, China}

\begin{document}

\begin{abstract}
Consider the following stochastic differential equation driven
by multiplicative noise on $\mathbb{R}^d$ with a superlinearly growing drift coefficient,
\begin{align*}
    \mathrm{d} X_t = b (X_t) \, \mathrm{d} t + \sigma (X_t) \, \mathrm{d} B_t.
\end{align*}
It is known that the corresponding explicit Euler schemes may not converge. In this article, we analyze an explicit and easily implementable numerical method for approximating such a stochastic differential equation, i.e.\ its tamed Euler-Maruyama approximation. Under partial dissipation conditions ensuring the ergodicity, we obtain the uniform-in-time convergence rates of the tamed Euler-Maruyama process under $L^{1}$-Wasserstein distance and total variation distance. 

\end{abstract}

\maketitle

\section{Introduction and Main Results}{\label{Results}}

Consider the following stochastic differential equation (SDE) on $\R^d$:
\begin{align} \label{SDE}
    \dif X_t = b (X_t) \, \dif t + \sigma (X_t) \, \dif B_t, \quad X_0=x_0,
\end{align}
where $b \colon \mathbb{R}^d \rightarrow \mathbb{R}^d$ is a function satisfying polynomial growth, $\sigma \colon \R^d \to \R^{d \times d}$, and $(B_t)_{t \geq 0}$ denotes the $d$-dimensional Brownian motion in a probability space $\left(\Omega, \mathscr{F},\left( \mathscr{F}_t \right)_{t \geqslant 0}, \mathbb{P}\right)$.

It is well-known that the corresponding explicit Euler-Maruyama (EM) schemes of SDEs \eqref{SDE} may not converge with respect to $L^{1}$-Wasserstein distance 
when the drift coefficients are allowed to grow super-linearly; see, for example, \cite[Theorem 2.1]{hutzenthaler2011strong}. As a consequence, many modified EM schemes have been introduced for such SDEs over the past decade, including tamed EM schemes \cite{Sabanis2013EulerAW,sabanis2013note}, adaptive EM  schemes \cite{do2024tamed,fang2017adaptive,giles2020adaptive,lamba2007adaptive}, truncated EM schemes \cite{mao2015truncated,song2022strong}, and implicit EM schemes \cite{mao2013strong}.

We consider a tamed Euler-Maruyama approximation based on Newton method  to numerically approximate SDE \eqref{SDE}: 
\begin{align} \label{EM1}
    Y_{t_{n+1}} = Y_{t_n} +  \frac{b (Y_{t_n})}{1 + \eta_{n+1}^\alpha \opnorm{\nabla b (Y_{t_n})}}\eta_{n+1} + \sigma (Y_{t_n}) (B_{t_{n+1}} - B_{t_n}), \quad n \geq 0,
\end{align}
with $Y_0 = X_0=x_0$, where $\alpha \in (0, 1 / 2)$ is a constant, $\opnorm{\cdot}$ denotes the operator norm, $\{ \eta_n \}_{n \geq 1}$ is a sequence of step sizes, $t_0 := 0$, and $t_n := \sum_{k = 1}^n \eta_k$. The associated continuous time Euler-Maruyama Scheme  of \eqref{EM1} is defined as
\begin{align} \label{EM2}
    \dif Y_t = \frac{b (Y_{t_n})}{1 + \eta_{n+1}^\alpha \opnorm{\nabla b (Y_{t_n})}} \, \dif t + \sigma (Y_{t_n}) \, \dif B_t, \quad t \in [t_n, t_{n+1}], \quad n \geq 0.
\end{align}

In this paper, we aim to study the convergence rate of the tamed Euler-Maruyama process  \eqref{EM1} for large time under $L^{1}$-Wasserstein distance and total variation distance, i.e.
$$
\mathbb{W}_1\left(\LL (X_{t_n} ), \LL (Y_{t_n} )\right), \; \dtv \left(\LL (X_{t_n} ), \LL (Y_{t_n} )\right) \rightarrow 0 \text {  as  } n \rightarrow \infty,
$$
where $\LL(\xi)$ is the distribution of a random variable $\xi$. For $\Pi\left(\mu, \nu\right)$ being the class of all couplings of probability measures $\mu, \nu$ on $\mathbb{R}^{d}$, the $L^{1}$-Wasserstein distance is defined as
\begin{align*}
    \W_{1}(\mu,\nu):=\inf_{\pi\in \Pi (\mu,\nu)} \left\{ \int_{\mathbb{R}^{d}\times \mathbb{R}^{d}}\abs{x-y} \pi(\mathrm{d}x,\mathrm{d}y) \right\},
\end{align*}
while the total variation distance between them is given by 
\begin{align*}
    \dtv(\mu,\nu):=\inf_{\pi\in \Pi (\mu,\nu)} \left\{ \int_{\mathbb{R}^{d}\times \mathbb{R}^{d}} \mathbf{1}_{\{x\neq y\}} \pi(\mathrm{d}x,\mathrm{d}y) \right\}.
\end{align*}
It is well known that, by Kantorovich-Rubinstein theorem \cite{Villani2003},
\begin{align*}
    \W_{1}(\mu,\nu)=\sup_{f\in \mathrm{Lip}(1)} \abs{ \int_{\R^d} f(x) \mu (\dif x) - \int_{\R^d} f(x) \nu (\dif x)},
\end{align*}
and
\begin{align} \label{e:KR-TV}
    \dtv(\mu,\nu)= \frac{1}{2} \sup_{f\in \mathcal{B}_{b}(\R^d), \, \norm{f}_\infty \leq 1} \abs{ \int_{\R^d} f(x) \mu (\dif x) - \int_{\R^d} f(x) \nu (\dif x)},
\end{align}
where $\mathrm{Lip}(1)=\left\{h: \mathbb{R}^d \rightarrow \mathbb{R} ;|h(y)-h(x)| \leq|y-x|\right\}$.

This paper uses tamed Euler-Maruyama approximation to approximate the SDEs with non-globally Lipschitz drift for large time under the $L^{1}$-Wasserstein distance and total variation distance. As we know, \cite[Theorem 2]{Sabanis2013EulerAW} shows that explicit schemes \eqref{EM1} converge in $L^p$ to the solution of the corresponding SDEs \eqref{SDE} in finite time, where the value of $p$ is related to the order of the drift term. In contrast, our paper analyzes the long-term behavior of scheme \eqref{EM1}, and the scheme is applicable to more general variable step sizes. The core methods of this paper are domino decomposition and Malliavin analysis methods.

The paper is organized as follows. In the rest of Section \ref{Results}, under certain assumptions, we provide the estimates for $\mathbb{W}_1 (\LL (X_{t_n}), \LL (Y_{t_n}) )$ and $\dtv (\LL (X_{t_n}), \LL (Y_{t_n}) )$. In Section \ref{proof1}, we present the lemmas required in the proof of the main theorem, including gradient estimates, moment estimates, and one step error estimates. In Section \ref{proof}, we present the proof of the main theorem. In the appendix, we provide proofs for some technical lemmas in Section \ref{proof1} and \ref{proof}.

\subsection{Notations} Throughout the paper, $\R^d$ denotes the $d$-dimensional Euclidean space, with norm $\abs{\cdot}$ and scalar product $\langle \cdot, \cdot \rangle$. The open ball centered at $x \in \R^d$ with a radius of $R > 0$ is denoted by $B (x, R) = \{ y \in \R^d \colon \abs{y - x} < R \}$.
For $q, s \in \R$, we denote $q \lor s = \max \{q, s\}$ and $q \land s = \min \{q, s\}$. 

The operator norm of a tensor $A = (a_{i_1 \cdots i_\kappa})_{i_1, \dots, i_\kappa = 1}^d \in \R^{d^{\otimes \kappa}}$, $\kappa = 1, 2, \dots$ is denoted by
\begin{align*}
    \opnorm{A} := \sup \left\{ \sum_{i_1, \dots, i_\kappa = 1}^d a_{i_1 \cdots i_\kappa} v_{i_1}^{(1)} \cdots v_{i_\kappa}^{(\kappa)} \colon v^{(1)}, \dots, v^{(\kappa)} \in \R^d, \, \abs{v^{(1)}} =  \dots = \abs{v^{(\kappa)}} = 1 \right\}.
\end{align*}

For $\kappa, r = 1, 2, \dots$, the set of bounded measurable tensor-valued functions $f \colon \mathbb{R}^{d} \to \R^{d^{\otimes \kappa}}$ is denoted by $\mathcal{B}_{b} (\R^d; \R^{d^{\otimes \kappa}})$, and the set of functions with $r$-th continuously differentiable components is denoted by $\mathcal{C}^r (\R^d; \R^{d^{\otimes \kappa}})$. Given $f = (f_{i_1 \cdots i_\kappa})_{i_1, \dots, i_\kappa = 1}^d \in \mathcal{C}^1 (\R^d; \R^{d^{\otimes \kappa}})$ and $v \in \R^d$, we denote
\begin{align*}
    \nabla_v f \colon \mathbb{R}^{d} &\longrightarrow \phantom{xxxxx} \R^{d^{\otimes \kappa}}, \\
    x \phantom{x} &\longmapsto ( \sca{\nabla f_{i_1 \cdots i_\kappa} (x), v} )_{i_1, \dots, i_\kappa = 1}^d.
\end{align*}
For $f \in \mathcal{C}^r (\R^d; \R^{d^{\otimes \kappa}})$, we further denote
\begin{align*}
    \oinorm{\nabla^r f} := \sup \left\{ \opnorm{\nabla_{v_1} \dots \nabla_{v_r} f (x)} \colon x, v_1, \dots, v_r \in \R^d; \; \abs{v_1}, \dots, \abs{v_r} \leq 1 \right\},
\end{align*}
and
\begin{align*}
    \mathcal{C}_{b}^r (\R^d; \R^{d^{\otimes \kappa}}) := \left\{ f \in \mathcal{C}^r (\R^d; \R^{d^{\otimes \kappa}}) \colon \oinorm{f}, \oinorm{\nabla f}, \dots, \oinorm{\nabla^r f} < + \infty \right\}.
\end{align*}
Especially, $\oinorm{f} := \sup \{ \opnorm{f (x)} \colon x \in \R^d \}$ for function $f \colon \R^d \to \R^{d^{\otimes \kappa}}$, and $\mathcal{B}_{b} (\R^d) = \mathcal{B}_{b} (\R^d; \R)$, $\mathcal{C}^r (\R^d) = \mathcal{C}^r (\R^d; \R)$, $\mathcal{C}_{b}^r (\R^d) = \mathcal{C}_{b}^r (\R^d; \R)$ for $\kappa = 0$. 

Whenever we want to emphasize the starting point $X_0=x$ for a given $x \in \mathbb{R}^d$, we will write $X_t^x$ instead of $X_t$; we use this also for $Y_k^y$ for a given $y \in \mathbb{R}^d$. Unless otherwise specified, the initial point of $X_t$ and $Y_k$ is assumed to be $x_0$. 

By $P_t, Q_k$ we denote the Markov semigroups of $X_t, Y_k$, respectively, i.e.
$$
P_t f(x)=P_{0,t}f(x)=\mathbb{E} f\left(X_t^{x}\right), \quad Q_k f({x})=Q_{0,k} f({x})=\mathbb{E} f\left(Y_k^{x}\right),
$$
for measurable function $f \colon \mathbb{R}^d \rightarrow \mathbb{R}$ belongs to the domain of $P_{t}$ and $Q_{k}$, $x \in \mathbb{R}^d, t \geq 0$, and $k=0,1,2, \cdots$.

Finally, we remark that $C$ denotes a positive constant which may be different even in a single chain of inequalities.

\subsection{Assumptions and main Results}
Throughout this paper, we introduce the following assumptions.
\begin{assumption} \label{A1}
    Assume $b \in \mathcal{C}^1 (\R^d; \R^d)$, and there exist constants $r \geq 0$ and $L_1, \lambda > 0$ such that for any $x, y \in \R^d$,
    \begin{gather}
        \sca{x, b (x)} \leq L_1 - \lambda \abs{x}^{r+2}, \label{eq:A11} \\
        \abs{b (x)} \leq L_1 (1 + \abs{x} \opnorm{\nabla b (x)}), \label{eq:A12} \\
        \abs{b (x) - b (y)} \leq L_1 \left(1 + \abs{x}^r + \abs{y}^r \right) \abs{x - y}. \label{eq:A13}
    \end{gather}
\end{assumption}

\begin{assumption} \label{A2}
    Assume $\sigma \in \mathcal{C}^2 (\R^d; \R^{d \times d})$, and there exists a constant $L_2$ such that
    \begin{align*}
        \oinorm{\sigma} \lor \oinorm{\sigma^{-1}} \leq L_2, \qquad
        \oinorm{\nabla \sigma} \leq L_2, \qquad
        \oinorm{\nabla^2 \sigma} \leq L_2.
    \end{align*}
\end{assumption}

According to \eqref{eq:A13}, we have $\opnorm{\nabla b (x)} \leq 2 L_1 (1 + \abs{x}^r)$, $\forall x \in \R^d$. Since $\nabla (\sigma^{-1}) = - \sigma^{-1} (\nabla \sigma) \sigma^{-1}$, Assumption \ref{A2} implies that $\oinorm{\nabla (\sigma^{-1})} \leq L_2^3$. 

Under the above assumptions, the SDE \eqref{SDE} is known to have a unique strong solution; see, for example, \cite[Theorem 3.3.1]{prevot2007concise}. 

In practical applications, the step size typically varies with each iteration. To control its behavior, an additional assumption on $\eta_n$ is necessary.
\begin{assumption} \label{A3}
    The sequence of step sizes $\{ \eta_n \}_{n \geq 1}$ is a non-increasing and positive sequence satisfying the following conditions:
    \begin{align*}
        \lim_{n \to \infty} \eta_n = 0, \quad
        \sum_{n = 1}^\infty \eta_n = + \infty, \quad
        \text{and}\quad\eta_{n-1} - \eta_n \leq \theta \eta_n^2, \quad \forall n \geq 2,
    \end{align*}
    for some $\theta > 0$.
\end{assumption}
A typical example is $\eta_n=\eta / n^\gamma$ for some constants $\eta>0$ and $\gamma \in(0,1]$.

Under Assumptions \ref{A1}, \ref{A2} and \ref{A3}, we establish Theorem \ref{thm1} and \ref{thm2}, which show the convergence rate of the tamed EM scheme \eqref{EM1} for large time under the $L^{1}$-Wasserstein distance and the total variation distance. The proofs will be given in Section \ref{proof}.
\begin{theorem} \label{thm1}
    Let $\left(X_t\right)_{t \geq 0}$ and $\left({Y}_k\right)_{k \geq 0}$ be defined by \eqref{SDE} and \eqref{EM1}. Suppose Assumption \ref{A1}, \ref{A2}, and \ref{A3} hold with $\eta_1 \leq \eta$ and $\theta \le  \theta_0$, and $b \in \mathcal{C}^2 (\R^d; \R^d)$ satisfies
    \begin{align*}
        \opnorm{\nabla^2 b (x)} \leq L_1 (1 + \abs{x}^r), \quad \forall x \in \R^d.
    \end{align*}
    Then for any $\alpha \in (0, 1 / 2)$, there exists a constant $C > 0$ such that,  
    \begin{align*}
        \W_1 (\LL (X_{t_n}), \LL (Y_{t_n})) &\leq C \eta_n^\alpha, \quad \forall n \geq 1, \\
        \dtv (\LL (X_{t_n}), \LL (Y_{t_n})) &\leq C \eta_n^\alpha, \quad \forall n \geq 1,
    \end{align*}
    where $C$, $\eta>0$, and $\theta_0>0$ only depend on $x$, $d$, $r$, $L_1$, $L_2$, and $\alpha$.
\end{theorem}

For the case $\sigma \equiv \sigma_0 \in \R^{d \times d}$, we have the following conclusion.
\begin{theorem}[Additive case] \label{thm2}
    Let $\left(X_t\right)_{t \geq 0}$ and $\left({Y}_k\right)_{k \geq 0}$ be defined by \eqref{SDE} and \eqref{EM1}. Suppose Assumption \ref{A1}, \ref{A2}, and \ref{A3} hold with $\sigma \equiv \sigma_0 \in \R^{d \times d}$, $\eta_1 \leq \eta$ and $\theta = \theta_0$. 
    
    Then for any $\alpha \in (0, 1 / 2)$, there exists a constant  $C > 0$ such that,  
    \begin{align*}
        \W_1 (\LL (X_{t_n}), \LL (Y_{t_n})) &\leq C \eta_n^\alpha, \quad \forall n \geq 1, \\
        \dtv (\LL (X_{t_n}), \LL (Y_{t_n})) &\leq C \eta_n^\alpha, \quad \forall n \geq 1,
    \end{align*}
    where $C$, $\eta>0$, and $\theta_0>0$ only depend on $x$, $d$, $r$, $L_1$, $L_2$, and $\alpha$.
\end{theorem}

\section{Auxiliary Lemmas}\label{proof1}
In this section, we provide some useful auxiliary lemmas for proving main theorems, including moment estimates and one step error estimates for $\left(X_t\right)_{t \geq 0}$, $\left(Y_k\right)_{k \geq 0}$, and gradient estimates for the Markov semigroups of $\left(X_t\right)_{t \geq 0}$.

We will frequently use the smooth function $V \colon \R^d \to [1, + \infty)$ such that, 
\begin{align}{\label{Lyapu}}
V (x) = \eup^{\abs{x}}, \quad \text{for } x\in \R^d \setminus B (\mathbf{0}, 1).
\end{align}
\subsection{Moment estimates}
In this section, we provide the moment estimators for $\left(X_t\right)_{t \geq 0}$ and $\left(Y_k\right)_{k \geq 0}$, as given in Lemma \ref{le:Xmoment} and Lemma \ref{le:Ymoment} below.
\begin{lemma}[Moment estimates for $X_t$]\label{le:Xmoment}
    Suppose Assumption \ref{A1} and \ref{A2} hold. For any $p\ge 1$, there exists a constant $C_{p} > 0$ not depending on $t$ such that
    \begin{align*}
        \E \left[V (X_t)^{p}\right] \leq \eup^{-\lambda t} \E\left[V (X_0)^{p}\right] + C_{p}, \quad \forall t \geq 0,
    \end{align*}
    and 
    \begin{align*}
        \E \abs{X_{t}}^{p}\le \eup^{-\lambda t} \E \abs{X_0}^p + C_{p}, \quad \forall t\ge 0.
    \end{align*}
    where $V(x)$ is a smooth function defined in \eqref{Lyapu}.
\end{lemma}
\begin{proof}
    Since $V$ is smooth, without loss of generality, we assume 
    \begin{align}{\label{Gra-V1}}
        \sup_{x\in B (\mathbf{0}, 1)}\opnorm{\nabla^{\kappa} V(x)}\le c_{1}, \quad \text{and}\quad \sup_{x\in B (\mathbf{0}, 1)} V(x)\le c_{1},
    \end{align}
    for $\kappa=1,2$ and some $c_{1}>0$. Notice 
    \begin{align}{\label{DoubleGra-V}}
        \opnorm{\nabla^{2} V(x)}=\opnorm{\frac{1}{|x|}I_{d}+\frac{xx^{T}}{|x|^{2}}-\frac{xx^{T}}{|x|^{3}} }V(x)
        \le 3V(x), \quad \forall |x|\ge 1,
    \end{align}
    where $I_{d}$ is the $d\times d$ identity matrix. 
    
    Hence, for $\tilde{V}_{p}(x) := V(x)^{p}$, it can be easily verified that, for $\kappa=1,2$,
    \begin{align*}
        \sup_{x\in B (\mathbf{0}, 1)}\opnorm{\nabla^{\kappa} \tilde{V}_{p}(x)}\le p c_{1}^{p}, \quad \text{and}\quad \sup_{x\in B (\mathbf{0}, 1)} \tilde{V}_{p}(x)\le c_{1}^{p},
    \end{align*}
    and
    \begin{align*}
        \opnorm{\nabla^{2} \Tilde{V}_{p}(x)} \le 3p^{2}\Tilde{V}_{p}(x), \quad \forall |x|\ge 1.
    \end{align*}
    It follows from Itô's formula, Assumption \ref{A1} and \ref{A2} that
    \begin{align*}
        &\dif \tilde{V}_{p}(X_{t})=\left[\langle\nabla \tilde{V}_{p}(X_{t}),b(X_{t})  \rangle+\frac{1}{2}\langle \nabla^{2} \tilde{V}_{p}(X_{t}), \sigma(X_{t}) \sigma(X_{t})^{T}
         \rangle_{\mathrm{HS}}\right] \dif t+\dif M_{t}\\
         &=\left[\frac{\tilde{V}_{p}(X_{t})}{|X_{t}|}\langle X_{t},b(X_{t})  \rangle+\frac{1}{2}\langle \nabla^{2} \tilde{V}_{p}(X_{t}), \sigma(X_{t}) \sigma(X_{t})^{T}
         \rangle_{\mathrm{HS}}\right] \mathbf{1}_{\{|X_{t}|\ge 1\}} \dif t\\
         &\quad +\left[\langle\nabla \tilde{V}_{p}(X_{t}),b(X_{t})  \rangle+\frac{1}{2}\langle \nabla^{2} \tilde{V}_{p}(X_{t}), \sigma(X_{t}) \sigma(X_{t})^{T}
         \rangle_{\mathrm{HS}}\right] \mathbf{1}_{\{|X_{t}| < 1\}} \dif t+\dif M_{t}\\
         &\le \left[\frac{L_{1}}{|X_{t}|}-\lambda |X_{t}|^{1+r}+\frac{3p^{2}}{2}\norm{\sigma(X_{t}) \sigma(X_{t})^{T}}_{\mathrm{HS}}\right]\tilde{V}_{p}(X_{t}) \mathbf{1}_{\{|X_{t}|\ge 1\}} \dif t\\
         &\quad +(p+1)c_{1}^{p}\left[|b(X_{t})|+\frac{1}{2} \norm{\sigma(X_{t}) \sigma(X_{t})^{T}}_{\mathrm{HS}}
         \right] \mathbf{1}_{\{|X_{t}| < 1\}} \dif t+\dif M_{t}\\
         &\le \left[-\lambda |X_{t}|^{1+r}+c_{2} \right]\tilde{V}_{p}(X_{t}) \dif t+\dif M_{t}\\
         &\le [-\lambda \tilde{V}_{p}(X_{t})+c_{3}]\dif t +\dif M_{t},
    \end{align*}
    where the last inequality is obtained by choosing a large enough $c_{3}$ such that $(-\lambda |x|^{1+r}+c_{2})\Tilde{V}_{p}(x)\le -\lambda \Tilde{V}_{p}(x)+c_{3}$ holds for any $x\in \mathbb{R}^{d}$ and $M_{t}$ is the martingale term. The proof of the first result is completed by taking the expectation on both side and then using the Grönwall's inequality.\par
    The second result can be proved analogously, so we omit the proof.
\end{proof}
Before providing the moment estimates for $Y_{t_n}$, we state the following useful lemma first, which will be proved in Appendix \ref{appendix}.
\begin{lemma} \label{le:A2}
    For a $d$-dimensional random vector with non-degenerate Gaussian distribution $\xi \sim \mathcal{N} (\mu, \eta \Sigma)$, if $\eta \opnorm{\Sigma} \leq 1 / 6$, there exists a constant $C > 0$ only depending on $\opnorm{\Sigma}$ and $d$, such that

    \noindent (i) $\E \left[ \eup^{\abs{\xi}} \mathbf{1}_{\R^d \setminus B (\mu, 1 / 3)} (\xi) \right] \leq C \eta \eup^{\abs{\mu}}$.

    \noindent (ii) $\E \left[ \eup^{\abs{\xi}} \mathbf{1}_{B (\mu, 1 / 3)} (\xi) \right] \leq \eup^{\abs{\mu} + C \eta}$ for $\abs{\mu} \geq 2 / 3$.
\end{lemma}

\begin{lemma}[Moment estimates for $Y_{t_n}$] \label{le:Ymoment}
    For any $\alpha \in (0, 1 / 2)$, there exist constants $C, \eta, \lambda' > 0$ not depending on $n$ such that, if Assumption \ref{A1}, \ref{A2}, and \ref{A3} hold with $\eta_1 \leq \eta$, we have
    \begin{align*}
        \E [V (Y_{t_n})^3] \leq \eup^{-\lambda' t_n} \E [V (Y_0)^3] + C, \quad \forall n \geq 0.
    \end{align*}
    where $V(x)$ is  a smooth function defined in \eqref{Lyapu}.
\end{lemma}
\begin{proof}
For the convenience of the proof, we define
   \begin{align}{\label{eq:U}}
        U (x) := \begin{cases}
            \eup^{3 \abs{x}}, & \abs{x} \geq \frac{1}{3}; \\
            9 \eup \abs{x}^2, & \abs{x} < \frac{1}{3}.
        \end{cases}
    \end{align}
    Since $\abs{U (x) - V (x)^3} \leq C$, the desired result is equivalent to
    \begin{align*}
        \E U (Y_{t_n}) \leq \eup^{-\lambda' t_n} \E U (Y_0) + C, \quad \forall n \geq 0,
    \end{align*}
    which follows from
    \begin{align} \label{eq:le2pr1}
        \E U (Y_{t_n}) \leq \eup^{-\lambda' \eta_n} \E U (Y_{t_{n-1}}) + C \eta_n, \quad \forall n \geq 1.
    \end{align}
    In fact, applying \eqref{eq:le2pr1} recursively implies that
    \begin{align*}
        \E U (Y_{t_n})
        &\leq \eup^{-\lambda' t_n} \E U (Y_0) + C \sum_{k = 1}^n \eta_k \eup^{-\lambda' (t_n - t_k)} \\
        &\leq \eup^{-\lambda' t_n} \E U (Y_0) + C \sum_{k = 1}^n (1 - \eup^{-\lambda' \eta_k}) \eup^{-\lambda' (t_n - t_k)}\\
        &\leq \eup^{-\lambda' t_n} \E U (Y_{0}) + C\eup^{-\lambda't_n} \int_{ 0}^{t_n}  \eup^{\lambda' x}\dif x\\
        &\leq \eup^{-\lambda' t_n} \E U (Y_0) + C.
    \end{align*}

    It remains to prove \eqref{eq:le2pr1}. Recall that 
    \begin{align*}
        Y_{t_{n}} = Y_{t_{n-1}} + \eta_{n} \frac{b (Y_{t_{n-1}})}{1 + \eta_{n}^\alpha \opnorm{\nabla b (Y_{t_{n-1}})}} + \sigma (Y_{t_{n-1}}) (B_{t_{n}} - B_{t_{n-1}}),
    \end{align*}
    so the conditional distribution of $Y_{t_n}$ with respect to $Y_{t_{n-1}}$ is the normal distribution $\mathcal{N} (\mu, \Sigma)$, where
    \begin{align*}
        \mu = Y_{t_{n-1}} + \eta_{n} \frac{b (Y_{t_{n-1}})}{1 + \eta_{n}^\alpha \opnorm{\nabla b (Y_{t_{n-1}})}}, \qquad
        \Sigma = \eta_n \sigma (Y_{t_{n-1}})\sigma (Y_{t_{n-1}})^{T}.
    \end{align*}

    By Assumption \ref{A1} and the fact that $\frac{x^{r}}{1 + L_1 \left( 1 + x^r \right)}\geq \frac{1}{ 1 + 2L_1 }$ for $x \geq 1$, we have
    \begin{gather} \label{eq:mu2}
    \begin{split}
        \abs{\mu}^2
        =& \abs{Y_{t_{n-1}}}^2 + \left( \frac{\eta_n \abs{b (Y_{t_{n-1}})}}{1 + \eta_{n}^\alpha \opnorm{\nabla b (Y_{t_{n-1}})}} \right)^2 + \frac{2 \eta_n \sca{Y_{t_{n-1}}, b (Y_{t_{n-1}})}}{1 + \eta_{n}^\alpha \opnorm{\nabla b (Y_{t_{n-1}})}} \\
        \leq& (1 + 2 L_1^2 \eta_n^{2 - 2 \alpha}) \abs{Y_{t_{n-1}}}^2 + 2 L_1^2 \eta_n^2 + 2 L_1 \eta_n - \frac{2 \lambda \eta_n \abs{Y_{t_{n-1}}}^{r+2}}{1 + L_1 \left( 1 + \abs{Y_{t_{n-1}}}^r \right)} \\
        \leq& \left[ 1 + 2 L_1^2 \eta_n^{2 - 2 \alpha} - \frac{2 \lambda \eta_n}{ 1 + 2L_1 } \right] \abs{Y_{t_{n-1}}}^2 + 2 L_1 \eta_n + 2 L_1^2 \eta_n^2\\
        &+\left(\frac{2 \lambda \eta_n}{ 1 + 2L_1 }-\frac{2 \lambda \eta_n \abs{Y_{t_{n-1}}}^{r}}{1 + L_1 \left( 1 + \abs{Y_{t_{n-1}}}^r \right)}\right)\abs{Y_{t_{n-1}}}^{2}\mathbf{1}_{\abs{Y_{t_{n-1}}}\leq 1}\\
        \leq& \left[ 1 + 2 L_1^2 \eta_n^{2 - 2 \alpha} - \frac{2 \lambda \eta_n}{ 1 + 2L_1 } \right] \abs{Y_{t_{n-1}}}^2 + 2 L_1 \eta_n + 2 L_1^2 \eta_n^2+\frac{2 \lambda \eta_n}{ 1 + 2L_1 }.
        \end{split}
    \end{gather}
    So there exist constants $C, \lambda' > 0$ such that,  for $\eta_n \leq \eta$ sufficiently small, 
    \begin{align} \label{eq:le2pr3}
        \abs{\mu} \leq (1 - \lambda' \eta_n) \abs{Y_{t_{n-1}}} + C \eta_n.
    \end{align}
    
    If $\lvert Y_{t_{n-1}} \rvert \geq 1 / 3$, By \eqref{eq:U}, we have
    \begin{align*}
    U (x) \leq \eup^{3 \abs{x}} \mathbf{1}_{\R^d \setminus B (\mu, 1 /9)} (x) + \eup^{3 \abs{x}}\mathbf{1}_{ B (\mu, 1 /9)} (x)=J_1+J_2.
    \end{align*}
    For the first term, according to Lemma \ref{le:A2}, we have
    \begin{align*}
        \E ( \eup^{3 \abs{Y_{t_n}}} \mathbf{1}_{\R^d \setminus B (\mu, 1 /9)} (Y_{t_n}) \vert Y_{t_{n-1}} )
        \leq  C \eta_n \eup^{3 \abs{\mu}}.
    \end{align*}
    For the second term, it follows from \eqref{eq:A12} that, for $\eta_n \leq \eta$ sufficiently small,
    \begin{align*}
        \abs{3 \mu}
        \geq 3 \abs{Y_{t_{n-1}}} - \frac{3 \eta_n \abs{b (Y_{t_{n-1}})}}{1 + \eta_{n}^\alpha \opnorm{\nabla b (Y_{t_{n-1}})}}
        \geq 3 \abs{Y_{t_{n-1}}} (1 - L_1 \eta_n^{1 - \alpha}) - 3 L_1 \eta_n
        \geq \frac{2}{3},
    \end{align*}
    According to Lemma \ref{le:A2}, we have
    \begin{align*}
        \E ( \eup^{3 \abs{Y_{t_n}}} \mathbf{1}_{ B (\mu, 1 /3)} (Y_{t_n}) \vert Y_{t_{n-1}} )
        \leq  \eup^{C \eta_n} \eup^{3 \abs{\mu}}.
    \end{align*}
So we get that, for $\lvert Y_{t_{n-1}} \rvert \geq 1 / 3$,
    \begin{align} \label{eq:le2pr2}
    \begin{split}
        \E ( U (Y_{t_n}) \vert Y_{t_{n-1}} )
        &\leq (C \eta_n + \eup^{C \eta_n}) \eup^{3 \abs{\mu}}\\
        &\leq  \eup^{3 (1 - \lambda' \eta_n) \lvert Y_{t_{n-1}} \rvert} \eup^{C \eta_n}\\
        &\leq (1 - \lambda' \eta_n) U (Y_{t_{n-1}}) + C \eta_n\\
        &\leq \eup^{- \lambda' \eta_n}U (Y_{t_{n-1}}) + C \eta_n, 
        \end{split}
    \end{align}
    where the second inequality follows from \eqref{eq:le2pr3} and the fact that $\eup^{2C \eta_n}\geq\eup^{C \eta_n}(1+C \eta_n)\geq\eup^{C \eta_n} +C \eta_n $, the last-to-second inequality follows from Young's inequality.

    If $\lvert Y_{t_{n-1}} \rvert < 1 / 3$, it follows from \eqref{eq:A12} that, for $\eta_n \leq \eta$ sufficiently small,
    \begin{align*}
        \abs{\mu}
        \leq \abs{Y_{t_{n-1}}} + \frac{\eta_n \abs{b (Y_{t_{n-1}})}}{1 + \eta_{n}^\alpha \opnorm{\nabla b (Y_{t_{n-1}})}}
        \leq (1 + L_1 \eta_n^{1 - \alpha}) \abs{Y_{t_{n-1}}} + L_1 \eta_n
        \leq \frac{4}{9},
    \end{align*}
    which implies that 
    \begin{align}{\label{J3J4}}
    U (x) \leq \eup^{3 \abs{x}} \mathbf{1}_{\R^d \setminus B (\mu, 1 / 9)} (x) + 9 \eup \abs{x}^2.
    \end{align} 
    For the first term, according to Lemma \ref{le:A2}, we have
    \begin{align*}
        \E ( \eup^{3 \abs{Y_{t_n}}} \mathbf{1}_{\R^d \setminus B (\mu, 1 / 9)} (Y_{t_n}) \vert Y_{t_{n-1}} )
        \leq C \eta_n.
    \end{align*}
    For the second term, assumption \ref{A2} and \eqref{eq:le2pr3} imply that
    \begin{align*}
        \E ( \abs{Y_{t_n}}^2 \vert Y_{t_{n-1}} )
        \leq \abs{\mu}^2 + L_2^2 \E \abs{B_{t_{n}} - B_{t_{n-1}}}^2
        \leq (1 - \lambda' \eta_n) \abs{Y_{t_{n-1}}}^2 + C \eta_n.
    \end{align*}
    So we get that, for $\lvert Y_{t_{n-1}} \rvert < 1 / 3$, 
    \begin{align} \label{eq:le2pr5}
    \begin{split}
        \E ( U (Y_{t_n}) \vert Y_{t_{n-1}} )
        &\leq 9 \eup (1 - \lambda' \eta_n) \abs{Y_{t_{n-1}}}^2 + C \eta_n\\
        &= (1 - \lambda' \eta_n) U (Y_{t_{n-1}}) + C \eta_n\\
        &\leq \eup^{- \lambda' \eta_n}U (Y_{t_{n-1}}) + C \eta_n.
        \end{split}
    \end{align}     

    Combining \eqref{eq:le2pr2}, \eqref{eq:le2pr5}, we can get the desired result.
\end{proof}

\subsection{One step error estimates}
 In this section, by Lemma \ref{le:Xmoment} and Lemma \ref{le:Ymoment}, we provide the moment estimates for the one step error of $\left(X_t\right)_{t \geq 0}$, and $\left(Y_k\right)_{k \geq 0}$, which is  given in Lemma \ref{le:onestep} below. 

For any $x\in \R^d$ and $k\in\mathbb Z^+,$ let $\{ Y_{t_k,t}^{x} \}_{t\in [t_k,t_{k+1}]}$ solve the SDE  
\begin{align}{\label{oneY}}
\dif Y_{t_k,t}^{x}= \frac{b (x)}{1 + \eta_{k+1}^\alpha \opnorm{\nabla b (x)}}\dif t+\sigma(x) \dif B_t,\ \ \ X_{t_k,t_k}^{x}=Y_{t_k,t_k}^{x}=x, \ \ \  t\in [t_k, t_{k+1}]. 
\end{align} 
Define 
  \begin{align}  \label{SQ}
  Q_{t_k, t_{k+1}}f(x):= \E[f(Y_{t_k,t_{k+1}}^{x})],\ \  Q_{t_k,t_n} := Q_{t_k, t_{k+1}} Q_{t_{k+1}, t_{k+2}}\cdots Q_{t_{n-1}, t_n},\ n\ge k+1.
  \end{align}   

Correspondingly, for any $s\ge 0$ and $x\in \R^d$, let $\{X_{s,t}^{x}\}_{t\ge s} $ solve the SDE
\begin{align}{\label{oneX}}
\dif X_{s,t}^{x}= b(X_{s,t}^{x})\dif t+\sigma(X_{s,t}^{x}) \dif B_t,\ \ X_{s,s}^{x}=x,\ \ t\ge s.
\end{align}
Then the Markov semigroup $P_t$ associated with \eqref{SDE} satisfies 
\begin{equation}\label{SP}   
P_{t-s}f(x)= P_{s,t}f(x):=\E[f(X_{s,t}^{x})],\ \  \ t\ge s\ge 0.
\end{equation}   
Let $Q_{0,0}=P_0$ be the identity operator.

\begin{lemma} \label{le:onestep}
    Suppose Assumption \ref{A1} and \ref{A2} hold.  
    
    (i) For any $p\ge 1$, there exists a constant $C_{p} > 0$ such that for any $n \geq 1$ and $t\in [t_{n-1},t_{n}]$, 
    \begin{align} \label{eq:onestepXY}
        \E\abs{X_{t_{n-1},t}^{x} -x}^{p}\le C_{p}\eta_{n}^{\frac{p}{2}}(1+|x|^{r+1})^{p}, \quad 
        \E\abs{Y_{t_{n-1},t}^{x}-x}^{p}\le C_{p}\eta_{n}^{\frac{p}{2}}(1+|x|^{r+1})^{p};
    \end{align}
    
    (ii) There exists a constant $C> 0$ such that for any $n \geq 1$ and $t\in [t_{n-1},t_{n}]$,  
    \begin{align} \label{eq:onestepX-Y}
        \E \abs{X_{t_{n-1},t}^{x} - Y_{t_{n-1},t}^{x}}^4 \leq C\eta_{n}^{4}(1+|x|^{2r+1})^{4}.
    \end{align}
    Furthermore, if $\sigma \equiv \sigma_0 \in \R^{d \times d}$, we have
    \begin{align*}
        \E \abs{X_{t_{n-1},t}^{x} - Y_{t_{n-1},t}^{x}}^4  \leq C \eta_n^{4 + 4 \alpha}(1+|x|^{2r+1})^{4} .
    \end{align*}
\end{lemma}

\begin{proof}
    (i) By Jensen's inequality, it suffices to consider $p\ge 2$. For any $t\in [t_{n-1},t_{n}]$, (\ref{oneX}) and Hölder's inequality imply that
    \begin{align*}
        \E\abs{ X_{t_{n-1},t}^{x} - x }^{p}
        &=\E \abs{\int_{t_{n-1}}^{t} b(X_{t_{n-1},s}^{x}) \dif s +\int_{t_{n-1}}^{t} \sigma(X_{t_{n-1},s}^{x}) \dif B_{s}    }^{p}\\
        &\le 2^{p-1} \E \abs{\int_{t_{n-1}}^{t} b(X_{t_{n-1},s}^{x}) \dif s   }^{p}+2^{p-1}\E\abs{\int_{t_{n-1}}^{t} \sigma(X_{t_{n-1},s}^{x}) \dif B_{s}    }^{p}\\
        &\le 2^{p-1} \eta_{n}^{p-1} \int_{t_{n-1}}^{t} \E\abs{b(X_{t_{n-1},s}^{x})}^{p} \dif s   +2^{p-1}\eta_{n}^{\frac{p}{2}-1}\int_{t_{n-1}}^{t} \E\norm{\sigma(X_{t_{n-1},s}^{x})}_{\mathrm{HS}}^{p} \dif s   
    \end{align*}
    where the first inequality is a consequence of the inequality $|A+B|^{p}\le 2^{p-1}\left(|A|^{p}+|B|^{p} \right)$. 
    
    It follows from Assumption \ref{A1}, \ref{A2} and Lemma \ref{le:Xmoment} that
    \begin{align*} 
        \E\abs{ X_{t_{n-1},t}^{x} - x }^{p}
        \le C_{p}\left[\eta_{n}^{p-1} \int_{t_{n-1}}^{t} \E \abs{X_{t_{n-1},s}^{x}}^{(r+1)p} \dif s   +\eta_{n}^{\frac{p}{2}}\right]
        \le C_{p}\eta_{n}^{\frac{p}{2}}(1+|x|^{r+1})^{p},
    \end{align*}
    holds for some positive constant $C_{p}$. 
    
    Now we turn to prove the second inequality in \eqref{eq:onestepXY}. Notice that for any $t\in [t_{n-1},t_{n}]$, 
    \begin{align*}
        Y_{t_{n-1},t}^{x}-x\sim \mathcal{N}\left(\frac{b (x)(t-t_{n-1})}{1 + \eta_{n}^\alpha \opnorm{\nabla b (x)}} , \sigma(x)\sigma(x)^{T}(t-t_{n-1})\right).
    \end{align*}
    So, as a consequence of Assumption \ref{A1} and \ref{A2}, we have
    \begin{align*}
        \E \abs{Y_{t_{n-1},t}^{x}-x}^{p}
        &\le 2^{p-1}(t-t_{n-1})^{p}\abs{\frac{b (x)}{1 + \eta_{n}^\alpha \opnorm{\nabla b (x)}}}^{p} +2^{p-1}(t-t_{n-1})^{\frac{p}{2}} \|\sigma(x)\sigma(x)^{T}\|_{\mathrm{HS}}^{\frac{p}{2}}\E|B_{1}|^{p}\\
        &\le C_{p}\left[(t-t_{n-1})^{p}(1+|x|^{r+1})^{p}+(t-t_{n-1})^{\frac{p}{2}}\right] \le C_{p}(t-t_{n-1})^{\frac{p}{2}}(1+|x|^{r+1})^{p}.
    \end{align*}
    
    (ii) It follows from Assumption \ref{A1} that, for any $y\in \R^{d}$,
    \begin{align} \label{eq:drift}
        \begin{split}
        &\mathrel{\phantom{=}} \abs{b (y)-\frac{b(x)}{1+\eta_{n}^{\alpha} \opnorm{\nabla b(x)}} }\\
        &\le \abs{b (y)-b(x)}+ \frac{\eta_{n}^{\alpha}\opnorm{\nabla b(x)}}{1+\eta_{n}^{\alpha} \opnorm{\nabla b(x)}} \abs{b(x)}\\
        &\le L_{1} (1+|x|^{r}+|y|^r) \abs{y-x}+C\eta_{n}^{\alpha}\left(1+|x|^{2r+1}\right).
        \end{split}
    \end{align}
    Together with \eqref{oneY} and Assumption \ref{A2}, we have for any $t\in (t_{n-1},t_{n}]$,
    \begin{align*}
        &\mathrel{\phantom{=}} \E\abs{ X_{t_{n-1},t}^{x} - Y_{t_{n-1},t}^{x} }^{4}\\
        &=\E \abs{ \int_{t_{n-1}}^{t} \left( b(X_{t_{n-1},s}^{x})- \frac{b(x)}{1+\eta_{n}^{\alpha} \opnorm{\nabla b(x)}} \right) \dif s +\int_{t_{n-1}}^{t} (\sigma(X_{t_{n-1},s}^{x})-\sigma(x)) \dif B_{s} }^{4}\\
        &\le 8\eta_{n}^{3} \int_{t_{n-1}}^{t} \E\abs{b(X_{t_{n-1},s}^{x})- \frac{b(x)}{1+\eta_{n}^{\alpha} \opnorm{\nabla b(x)}} }^{4}\dif s +8\eta_{n}   \int_{t_{n-1}}^{t} \E\abs{\sigma(X_{t_{n-1},s}^{x})-\sigma(x)}^{4} \dif s   \\
        &\le C\eta_{n}^{3}\left[\int_{t_{n-1}}^{t} \E\left[(1+|x|^{4r}+|X_{t_{n-1},t}^{x}|^{4r})\abs{X_{t_{n-1},t}^{x}-x}^{4} \right]\dif s +\eta_{n}^{4\alpha+1}(1+|x|^{2r+1})^{4}     \right]\\
        &\quad +C\eta_{n} \int_{t_{n-1}}^{t} \E\abs{X_{t_{n-1},s}^{x}-x}^{4} \dif s\\
        &\le C\eta_{n}^{4}(1+|x|^{2r+1})^{4}.
    \end{align*}
where the last inequality comes from Lemma \ref{le:Xmoment} and \eqref{eq:onestepXY}.
    
    If furthermore $\sigma \equiv \sigma_0 \in \R^{d \times d}$, then
    \begin{align*}
        \E\abs{ X_{t_{n-1},t}^{x} - Y_{t_{n-1},t}^{x} }^{4}\
        =\E \abs{ \int_{t_{n-1}}^{t} \left( b(X_{t_{n-1},s}^{x})- \frac{b(x)}{1+\eta_{n}^{\alpha} \opnorm{\nabla b(x)}} \right) \dif s }^{4}.
    \end{align*}
    So the result can be obtained using the same method and the proof is complete.
\end{proof}

\subsection{Gradient estimate for the semigroups of $X_t$}
In this section, we mainly use Lemma \ref{le:Xmoment}, \ref{le:onestep} and the Bismut–Elworthy–Li formula (see Lemma \ref{le:BEL} and \ref{le:BEL1} below) to provide gradient estimates for the Markov semigroups of $X_t$, which shows in Lemma \ref{le:gradient}. 

For any $v,w\in \R^{d}$ and fixed $t>0$, we can define
\begin{gather}  \label{def:R,K}
    \begin{split}
    R^{v}_{t}&:=\nabla_{v}X_{t}^{x}:= \lim_{\epsilon\to 0} \frac{X_{t}^{x+\epsilon v}-X_{t}^{x}}{\epsilon},\\
    K^{v,w}_{t}&:= \nabla_{v}\nabla_{w} X_{t}^{x}:= \lim_{\epsilon \to 0} \frac{\nabla_{v}X_{t}^{x+\epsilon w}-\nabla_{v}X_{t}^{x}}{\epsilon}. 
\end{split}
\end{gather}
Combining above definitions with \eqref{SDE}, it is not difficult to see that $R^{v}_{t}$ and $K^{v,w}_{t}$ solve the following equations:
\begin{align}{\label{Rv}}
    \dif R_t^v = \nabla_{R_t^v} b (X_t^x) \, \dif t + \nabla_{R_t^v} \sigma (X_t^x) \, \dif B_t,\quad R_0^v = v
\end{align}
and
\begin{gather} {\label{Kvw}}
    \begin{split}
    \dif K_t^{v,w} = &\left( \nabla_{K_t^{v,w}} b (X_t^x) + \nabla_{R_t^v} \nabla_{R_t^w} b (X_t^x) \right) \dif t \\
    &+ \left( \nabla_{K_t^{v,w}} \sigma (X_t^x) + \nabla_{R_t^v} \nabla_{R_t^w} \sigma (X_t^x) \right) \dif B_t,\quad K_{0}^{v,w}=0. 
    \end{split}
\end{gather}

The proof of following Bismut–Elworthy–Li formula is standard and classical. We refer to \cite{Bi84, EL94} for more details. 
\begin{lemma}[Bismut–Elworthy–Li formula] \label{le:BEL}
    Let $\{X_{t}\}_{t\ge 0}$ be the solution of \eqref{SDE}. Then for any $t > 0$, $ v \in \R^d$ and $f \in \mathcal{C}_b^1 (\R^d)$, we have
    \begin{align} \label{eq:BEL}
        \nabla_v P_t f (x) = \frac{1}{t} \E \left[ f (X_t^x) \int_0^t \sca{\sigma^{-1} (X_t^x) R_t^v, \dif B_t} \right].
    \end{align}
  \end{lemma}  
  \begin{lemma}  \label{le:BEL1}
   Let $\{X_{t}\}_{t\ge 0}$ be the solution of \eqref{SDE}.  Suppose Assumption \ref{A1} and \ref{A2} hold. Then for any $p\ge 2$, 
     
    \noindent (i) There exists a constant $C> 0$ such that 
    \begin{align}\label{ineq:Rmoment1}
    \E \abs{R_t^v}^p \leq \eup^{Ct} \abs{v}^p V (x), \qquad \forall t > 0.
\end{align}

\noindent (ii) Further assume $b \in \mathcal{C}^2 (\R^d; \R^d)$ and $\opnorm{\nabla^2 b (x)} \leq L_1 (1 + \abs{x}^r)$, $\forall x \in \R^d$, there exists a constant $C> 0$ such that 
        \begin{align}\label{ineq:Rmoment2}
    \E \abs{K_t^{v,w}}^p \leq \eup^{Ct} \abs{v}^p\abs{w}^p V (x), \qquad \forall t > 0,
\end{align}
where $V(x)$ is a smooth function defined in \eqref{Lyapu}.
\end{lemma}   

\begin{proof}
    (i) 
    By \eqref{Rv}, \eqref{SDE} and It\^o's formula, we have
    \begin{gather} {\label{RtV}}
    \begin{split}
        \dif \abs{R_t^v}^p=&  p|R_t^v|^{p-2}\langle R_t^v,\nabla_{R_t^v} b (X_t^x) \rangle \dif t +\frac{1}{2}p(p-2)|R_t^v|^{p-4}|R_t^v \nabla_{R_t^v} \sigma (X_t^x)|^2 \dif t \\ 
        &+\frac12p|R_t^v|^{p-2}\norm{\nabla_{R_t^v} \sigma (X_t^x)}_{\mathrm{HS}}^{2} \dif t+ p|R_t^v|^{p-2}\langle R_t^v, \nabla_{R_t^v} \sigma (X_t^x) \dif B_{t}\rangle,
        \end{split}
    \end{gather}
    and
    \begin{gather}{\label{VXt}}
    \begin{split}
        \dif V(X_{t}^{x})= &\langle \nabla V(X_{t}^{x}), b(X_{t}^{x}) \rangle \dif t+\frac{1}{2}\langle \nabla^{2} V(X_{t}^{x}), \sigma(X_{t}^{x})\sigma(X_{t}^{x})^{T}  \rangle_{\mathrm{HS}}\dif t\\
        &+\langle \nabla V(X_{t}^{x}), \sigma(X_{t}^{x}) \dif B_t\rangle.
        \end{split}
    \end{gather}
    It follows from \eqref{RtV} and \eqref{VXt} that
    \begin{gather} \label{eq:ItoRV}
        \begin{split}
            &\dif (\abs{R_t^v}^p  V(X_{t}^{x}))\\
        =&\bigg{[}p|R_t^v|^{p-2}V(X_{t}^{x})\langle R_t^v,\nabla_{R_t^v} b (X_t^x) \rangle+\frac{p}{2}(p-2)V(X_{t}^{x})|R_t^v|^{p-4}|R_t^v \nabla_{R_t^v} \sigma (X_t^x)|^2 \\
        &+\frac12pV(X_{t}^{x})|R_t^v|^{p-2}\norm{\nabla_{R_t^v} \sigma (X_t^x)}_{\mathrm{HS}}^{2}+\abs{R_t^v}^p\langle \nabla V(X_{t}^{x}), b(X_{t}^{x}) \rangle   \\
        & +\frac{\abs{R_t^v}^p}{2}\langle \nabla^{2} V(X_{t}^{x}), \sigma(X_{t}^{x})\sigma(X_{t}^{x})^{T}  \rangle_{\mathrm{HS}}\\
        & +p|R_t^v|^{p-2}\langle \nabla_{R_t^v} \sigma (X_t^x) R_t^v,  \sigma(X_{t}^{x})\nabla V(X_{t}^{x})  \rangle \bigg{]} \dif t +\dif M_{t},
        \end{split}
    \end{gather}
    where $M_{t}$ is the martingale term. 
    
    For any $x\in \mathbb{R}^{d}$, by \eqref{Gra-V1}, \eqref{DoubleGra-V}, we know there exists some constant $c$ such that
    \begin{align*}
        \abs{\nabla V(x)}\le cV(x),\quad \text{and}\quad \opnorm{\nabla^{2} V(x)}\le cV(x).
    \end{align*}
    Together with Assumption \ref{A1} and \ref{A2}, and the fact that $\|\nabla_{R_t^v} \sigma (X_t^x)\|_{\mathrm{op}} \leq\|\nabla_{R_t^v} \sigma (X_t^x)\|_{\mathrm{HS}}$, we have the estimates for the first three terms in the right side of \eqref{eq:ItoRV}, i.e.
    \begin{align*}
        &p|R_t^v|^{p-2}V(X_{t}^{x})\langle R_t^v,\nabla_{R_t^v} b (X_t^x) \rangle+\frac{1}{2}p(p-2)V(X_{t}^{x})|R_t^v|^{p-4}|R_t^v \nabla_{R_t^v} \sigma (X_t^x)|^2\\
        &+\frac12pV(X_{t}^{x})|R_t^v|^{p-2}\norm{\nabla_{R_t^v} \sigma (X_t^x)}_{\mathrm{HS}}^{2}\\
        \le&p|R_t^v|^{p-2}V(X_{t}^{x})\langle R_t^v,\nabla_{R_t^v} b (X_t^x) \rangle+\frac{1}{2}p(p-1)V(X_{t}^{x})|R_t^v|^{p-2}\norm{\nabla_{R_t^v} \sigma (X_t^x)}_{\mathrm{HS}}^{2}\\
        \le &p \abs{R_t^v}^pV(X_{t}^{x}) \opnorm{\nabla b(X_{t}^{x})}+\frac{1}{2}p(p-1)dL_{2}^{2}\abs{R_t^v}^pV(X_{t}^{x})\\
        \le& C \abs{R_t^v}^pV(X_{t}^{x}) \left( 1+|X_{t}^{x}|^{r} \right).
    \end{align*}
    Further notice that for $|x|\ge 1$, $\nabla V(x)=\frac{x}{|x|} V(x)$, then, we have
    \begin{align*}
        &\abs{R_t^v}^p\langle \nabla V(X_{t}^{x}), b(X_{t}^{x}) \rangle+\frac{\abs{R_t^v}^p}{2}\langle \nabla^{2} V(X_{t}^{x}), \sigma(X_{t}^{x})\sigma(X_{t}^{x})^{T}  \rangle_{\mathrm{HS}}\\
        \le &\abs{R_t^v}^pV(X_{t}^{x}) \left[\left\langle \frac{X_{t}^{x}}{\abs{X_{t}^{x}}}, b(X_{t}^{x})  \right\rangle +\frac{cdL_{2}^{2}}{2}\right] \mathbf{1}_{\{|X_{t}^{x}|\ge 1\}}\\
        & +c \abs{R_t^v}^pV(X_{t}^{x})\left[ \abs{b(X_{t}^{x})}+\frac{dL_{2}^{2}}{2}  \right] \mathbf{1}_{\{|X^{x}_{t}|<1\}}\\
        \le &C\abs{R_t^v}^pV(X_{t}^{x}) \left[\left(L_{1}+\frac{cdL_{2}^{2}}{2}-\lambda |X_{t}^{x}|^{r+1}\right) \mathbf{1}_{\{|X_{t}^{x}|\ge 1\}} + \mathbf{1}_{\{|X_{t}^{x}|<1\}}\right],
    \end{align*}
    and
    \begin{align*}
        &p|R_t^v|^{p-2}\langle \nabla_{R_t^v} \sigma (X_t^x) R_t^v,  \sigma(X_{t}^{x})\nabla V(X_{t}^{x})  \rangle \\
        \le& cp\abs{R_t^v}^pV(X_{t}^{x}) \oinorm{\sigma} \oinorm{\nabla \sigma}\le C\abs{R_t^v}^pV(X_{t}^{x}).
    \end{align*}
    
    Combining all these estimates with \eqref{eq:ItoRV} gives
    \begin{align*}
        \dif \E (\abs{R_t^v}^p  V(X_{t}^{x}))
        &\le C \E \left(\abs{R_t^v}^p  V(X_{t}^{x}) \left[ \left(C'-\lambda|X_{t}^{x}|^{r+1}+|X_{t}^{x}|^{r} \right) \mathbf{1}_{\{|X_{t}^{x}|\ge 1\}} + \mathbf{1}_{\{|X_{t}^{x}|<1\}}   \right] \right)\\
        &\le C\E\left[\abs{R_t^v}^p  V(X_{t}^{x})\right].
    \end{align*}
    Since $V(x)\ge 1$ for any $x\in \R^{d}$, it follows from Grönwall's inequality that,
    \begin{align}{\label{resultRV}}
    \E \abs{R_t^v}^p \leq \E (\abs{R_t^v}^p  V(X_{t}^{x})) \leq  \eup^{Ct} \abs{v}^p V (x).
    \end{align}
    
    \noindent (ii) By \eqref{Kvw} and It\^o's formula,
    \begin{gather}{\label{KKvw}}
    \begin{split}
    \dif \abs{K_t^{v,w}}^p=&  p|K_t^{v,w}|^{p-2}\langle K_t^{v,w},\nabla_{K_t^{v,w}} b (X_t^x) + \nabla_{R_t^v} \nabla_{R_t^w} b (X_t^x)\rangle \dif t \\
    &+\frac{1}{2}p(p-2)|K_t^{v,w}|^{p-4}|K_t^{v,w} (\nabla_{K_t^{v,w}} \sigma (X_t^x)+ \nabla_{R_t^v} \nabla_{R_t^w} \sigma (X_t^x))|^2 \dif t  \\ 
        &+\frac12p|K_t^{v,w}|^{p-2}\norm{\nabla_{K_t^{v,w}} \sigma (X_t^x)+ \nabla_{R_t^v} \nabla_{R_t^w} \sigma (X_t^x)}_{\mathrm{HS}}^{2} \dif t \\
        &+ p|K_t^{v,w}|^{p-2}\langle K_t^{v,w}, (\nabla_{K_t^{v,w}} \sigma (X_t^x)+ \nabla_{R_t^v} \nabla_{R_t^w} \sigma (X_t^x)) \dif B_{t}\rangle. 
        \end{split}
    \end{gather}
    It follows from \eqref{KKvw} and \eqref{VXt} that
    \begin{align} \label{eq:ItoKV}
        \begin{aligned}
            &\dif (\abs{K_t^{v,w}}^p  V(X_{t}^{x}))
        =\bigg{[}p|K_t^{v,w}|^{p-2}V(X_{t}^{x})\langle K_t^{v,w},\nabla_{K_t^{v,w}} b (X_t^x) + \nabla_{R_t^v} \nabla_{R_t^w} b (X_t^x)\rangle  \\
    &\qquad+\frac{1}{2}p(p-2)V(X_{t}^{x})|K_t^{v,w}|^{p-4}|K_t^{v,w} (\nabla_{K_t^{v,w}} \sigma (X_t^x)+ \nabla_{R_t^v} \nabla_{R_t^w} \sigma (X_t^x))|^2 \\ 
        &\qquad+\frac12pV(X_{t}^{x})|K_t^{v,w}|^{p-2}\norm{\nabla_{K_t^{v,w}} \sigma (X_t^x)+ \nabla_{R_t^v} \nabla_{R_t^w} \sigma (X_t^x)}_{\mathrm{HS}}^{2}  \\
        &\qquad+\abs{K_t^{v,w}}^p\langle \nabla V(X_{t}^{x}), b(X_{t}^{x}) \rangle  +\frac{\abs{K_t^{v,w}}^p}{2}\langle \nabla^{2} V(X_{t}^{x}), \sigma(X_{t}^{x})\sigma(X_{t}^{x})^{T}  \rangle_{\mathrm{HS}}\\
        & \qquad+p|K_t^{v,w}|^{p-2}\langle K_t^{v,w} (\nabla_{K_t^{v,w}} \sigma (X_t^x)+ \nabla_{R_t^v} \nabla_{R_t^w} \sigma (X_t^x)),  \sigma(X_{t}^{x})\nabla V(X_{t}^{x})  \rangle \bigg{]} \dif t +\dif M_{t},
        \end{aligned}
    \end{align}
    where $M_{t}$ is the martingale term. 

    By Assumption \ref{A1} and \ref{A2}, and $\opnorm{\nabla^2 b (x)} \leq L_1 (1 + \abs{x}^r)$, we have 
\begin{align*}
&p|K_t^{v,w}|^{p-2}V(X_{t}^{x})\langle K_t^{v,w},\nabla_{K_t^{v,w}} b (X_t^x) + \nabla_{R_t^v} \nabla_{R_t^w} b (X_t^x)\rangle  \\
    & +\frac{1}{2}p(p-2)V(X_{t}^{x})|K_t^{v,w}|^{p-4}|K_t^{v,w} (\nabla_{K_t^{v,w}} \sigma (X_t^x)+ \nabla_{R_t^v} \nabla_{R_t^w} \sigma (X_t^x))|^2  \nonumber\\ 
        & +\frac12pV(X_{t}^{x})|K_t^{v,w}|^{p-2}\norm{\nabla_{K_t^{v,w}} \sigma (X_t^x)+ \nabla_{R_t^v} \nabla_{R_t^w} \sigma (X_t^x)}_{\mathrm{HS}}^{2}\\
        \le  &p|K_t^{v,w}|^{p-1}V(X_{t}^{x})(|K_t^{v,w}|  \opnorm{\nabla b (X_t^x)} + | {R_t^v}||  {R_t^w}| \opnorm{\nabla^2 b (X_t^x)} )  \\
        & +\frac12pV(X_{t}^{x})|K_t^{v,w}|^{p-1} (|K_t^{v,w}|\norm{\nabla\sigma (X_t^x)}_{\mathrm{HS}}^{2}+ |{R_t^v}||R_t^w|\norm{\nabla^2 \sigma (X_t^x)}_{\mathrm{HS}}^{2})\\
        \le &C \abs{K_t^{v,w}}^{p-1}\left( |K_t^{v,w}|+ | {R_t^v}||  {R_t^w}|\right)V(X_{t}^{x}) \left( 1+|X_{t}^{x}|^{r} \right)\\
        \le &C \left( |K_t^{v,w}|^{p}+ (| {R_t^v}||  {R_t^w}|)^{p }\right)V(X_{t}^{x}) \left( 1+|X_{t}^{x}|^{r} \right).
    \end{align*}
    where the last inequality comes from Young's inequality. 

    Through calculations similar to those in (i), we have
    \begin{align*}
    &\abs{K_t^{v,w}}^p\langle \nabla V(X_{t}^{x}), b(X_{t}^{x}) \rangle  +\frac{\abs{K_t^{v,w}}^p}{2}\langle \nabla^{2} V(X_{t}^{x}), \sigma(X_{t}^{x})\sigma(X_{t}^{x})^{T}  \rangle_{\mathrm{HS}}\\
        \le &C\abs{K_t^{v,w}}^pV(X_{t}^{x}) \left[\left(L_{1}+\frac{cdL_{2}^{2}}{2}-\lambda |X_{t}^{x}|^{r+1}\right) \mathbf{1}_{\{|X_{t}^{x}|\ge 1\}} + \mathbf{1}_{\{|X_{t}^{x}|<1\}}\right],
    \end{align*}
    and
\begin{align*}
    &p|K_t^{v,w}|^{p-2}\langle K_t^{v,w} (\nabla_{K_t^{v,w}} \sigma (X_t^x)+ \nabla_{R_t^v} \nabla_{R_t^w} \sigma (X_t^x)),  \sigma(X_{t}^{x})\nabla V(X_{t}^{x})  \rangle\\
        \le& cp\abs{K_t^{v,w}}^{p-1}V(X_{t}^{x}) \oinorm{\sigma} (\abs{K_t^{v,w}}\oinorm{\nabla \sigma}+| {R_t^v}||  {R_t^w}|\oinorm{\nabla^2 \sigma})\\
        \le& C(\abs{K_t^{v,w}}^p+(| {R_t^v}||  {R_t^w}|)^{p })V(X_{t}^{x}).
    \end{align*}
    Combining all these estimates with \eqref{eq:ItoKV} and  the Cauchy-Schwarz inequality, we have 
     \begin{align*}
        &\E (\abs{K_t^{v,w}}^p  V(X_{t}^{x}))\\
        \le& C \E \left(\abs{K_t^{v,w}}^p  V(X_{t}^{x}) \left[ \left(C'-\lambda|X_{t}^{x}|^{r+1}+|X_{t}^{x}|^{r} \right) \mathbf{1}_{\{|X_{t}^{x}|\ge 1\}} + \mathbf{1}_{\{|X_{t}^{x}|<1\}}   \right] \right)\\
        &+\E\left[(| {R_t^v}||  {R_t^w}|)^{p } V(X_{t}^{x}) ( 1+|X_{t}^{x}|^{r} )\right]\\
        \le& C\E\left[\abs{K_t^{v,w}}^p  V(X_{t}^{x})\right]+\E\left[(| {R_t^v}||  {R_t^w}|)^{p } (V(X_{t}^{x}))^2\right]\\
        \le& C\E\left[\abs{K_t^{v,w}}^p  V(X_{t}^{x})\right]+ \left[\E | {R_t^v}| ^{2p } (V(X_{t}^{x}))^2\right]^{\frac12}\left[\E |  {R_t^w}| ^{2p } (V(X_{t}^{x}))^2\right]^{\frac12}.
    \end{align*}
    By using the same method as in the proof of \eqref{resultRV}, one can show that
    \begin{align*}
        \E \left[| {R_t^v}| ^{2p } (V(X_{t}^{x}))^2 \right]\le e^{Ct} |v|^{2p}\left(V(x)\right)^{2}, \quad \forall |v|\le 1.
    \end{align*}
    So it follows that
    \begin{align*}
         \E (\abs{K_t^{v,w}}^p  V(X_{t}^{x}))
        \le& C\E\left[\abs{K_t^{v,w}}^p  V(X_{t}^{x})\right]+ e^{C't}|v|^p|w|^p \left(V(x)\right)^{2}.
    \end{align*}
    Since $K_{0}^{v,w}=0$ and $V(x)\ge 1$, it follows from Grönwall's inequality that,
    \begin{align*} 
    \E \abs{K_t^{v,w}}^p \leq \E (\abs{K_t^{v,w}}^p  V(X_{t}^{x})) \le  \eup^{Ct} \abs{v}^p \abs{w}^p \left(V (x)\right)^{2} , \qquad \forall t > 0,
\end{align*}
Since this holds for any $p\ge 2$, by Hölder's inequality,
\begin{align*}
     \E \abs{K_t^{v,w}}^p \leq \sqrt{\E \abs{K_t^{v,w}}^{2p}}\le \sqrt{e^{Ct}|v|^{2p} |w|^{2p}V(x)^{2}}=e^{\frac{C}{2}t}|v|^p |w|^p V(x) , \qquad \forall t > 0,
\end{align*}
The proof is complete.
\end{proof}

We have the following property for the SDE \eqref{SDE}, which will be proved in Appendix \ref{appendix}.
\begin{lemma}{\label{ergodic}}
    Suppose Assumption \ref{A1} and \ref{A2} hold. Then the Markov semigroup $\{ P_t \}_{t \geq 0}$ is strongly Feller and irreducible, i.e.
    
    \noindent (a) For any $t > 0$ and $f \in \mathcal{B}_b (\R^d)$, $P_t f \in \mathcal{C}_b (\R^d)$.

    \noindent (b) For any $t > 0$, $x \in \R^d$ and nonempty open set $U \subseteq \R^d$, $P_t \mathbf{1}_{U} (x) > 0$.    
\end{lemma}

Combining Lemma \ref{le:BEL} and Lemma \ref{ergodic}, we can obtain the following gradient estimates.
\begin{lemma}[Gradient estimates] \label{le:gradient}
    Suppose Assumption \ref{A1} and \ref{A2} hold. There exist constants $C, c > 0$ such that
    
    \noindent (i) For any $t > 0$, $x \in \R^d$ and $f \in \mathcal{C}_b^1 (\R^d)$,
    \begin{align}
        \opnorm{\nabla P_t f (x)}
        &\leq \frac{C \eup^{-c t}}{\sqrt{t \land 1}} V (x) \norm{f}_\infty,  \label{eq:gra1_tot1}\\
        \opnorm{\nabla P_t f (x)}
        &\leq C \eup^{-c t} V (x) \oinorm{\nabla f}.\label{eq:gra1_tot2}
    \end{align}
    
    \noindent (ii) Further assume $b \in \mathcal{C}^2 (\R^d; \R^d)$ and $\opnorm{\nabla^2 b (x)} \leq L_1 (1 + \abs{x}^r)$, $\forall x \in \R^d$, then for any $t > 0$, $x \in \R^d$ and $f \in \mathcal{C}_b^2 (\R^d)$,
    \begin{align} 
        \opnorm{\nabla^2 P_t f (x)}
        &\leq \frac{C \eup^{-c t}}{t \land 1} V (x)^{\frac{3}{2}} \norm{f}_\infty, \label{eq:gra21}\\
        \opnorm{\nabla^2 P_t f (x)}
        &\leq \frac{C \eup^{-c t}}{\sqrt{t \land 1}} V (x)^{\frac{3}{2}} \oinorm{\nabla f}, \label{eq:gra22}
    \end{align}
     where $V(x)$ is the smooth function defined in \eqref{Lyapu}. 
\end{lemma}

\begin{proof}
\noindent (i) For $0 < t < 1$, Lemma \ref{le:BEL} and \ref{le:BEL1}, Assumption \ref{A2} show that for any $v \in \R^d$, $\abs{v} \leq 1$,
\begin{gather} \label{eq:le6pr1}
    \begin{split}
    \abs{\nabla_v P_t f (x)}
    &= \frac{1}{t} \abs{\E \left[ f (X_t^x) \int_0^t \sca{\sigma^{-1} (X_s^x) R_s^v, \dif B_s} \right]} \\
    &\leq \frac{1}{t} \norm{f}_\infty \sqrt{\E \abs{\int_0^t \sca{\sigma^{-1} (X_s^x) R_s^v, \dif B_s}}^2} \\
    &\leq \frac{C}{\sqrt{t}} \sqrt{V (x)} \norm{f}_\infty.
    \end{split}
\end{gather}
Combining Lemma \ref{le:onestep} and \ref{le:BEL}, and Assumption \ref{A2}, for any $v \in \R^d$, $\abs{v} \leq 1$, we have
\begin{gather} \label{eq:le6pr2}
    \begin{split}
    \abs{\nabla_v P_t f (x)}
    &= \frac{1}{t} \abs{\E \left[ \left( f (X_t^x) - f (x) \right) \int_0^t \sca{\sigma^{-1} (X_s^x) R_s^v, \dif B_s} \right]} \\
    &\leq \frac{1}{t} \oinorm{\nabla f} \sqrt{\E \abs{X_t^x - x}^2} \sqrt{\E \abs{\int_0^t \sca{\sigma^{-1} (X_s^x) R_s^v, \dif B_s}}^2} \\
    &\leq C (1 + \abs{x}^{r+1}) \sqrt{V (x)} \oinorm{\nabla f}.
    \end{split}
\end{gather}

Then we turn to the case $t \geq 1$. According to Lemma \ref{le:BEL},
\begin{align} \label{eq:Bismut}
   \begin{split}
    \nabla_v P_t f (x)
    &= \nabla_v P_1 (P_{t-1} f) (x) \\
    &= \E \left[ P_{t-1} f (X_1^x) \int_0^1 \sca{\sigma^{-1} (X_s^x) R_s^v, \dif B_s} \right] \\
    &= \E \left[ \left( P_{t-1} f (X_1^x) - \int_{\R^d} f (y) \mu (\dif y) \right) \int_0^1 \sca{\sigma^{-1} (X_s^x) R_s^v, \dif B_s} \right],
   \end{split}
\end{align}
where $\mu$ denotes the stationary distribution of $\{ X_t^x \}_{t \geq 0}$. It follows from Lemma \ref{le:Xmoment} that $\E \abs{X_t^x}^2 \leq \eup^{-\lambda t} \abs{x}^2 + C$, $\forall t > 0$, so Lemma \ref{ergodic} and \cite[Theorem 2.5 (a)]{BB2006} shows
\begin{align*}
    \abs{P_{t-1} f (X_1^x) - \int_{\R^d} f (y) \mu (\dif y)}
    \leq C \eup^{-ct} (1 + \abs{X_1^x}^2) \sup_{z \in \R^d} \frac{\abs{f (z)}}{1 + \abs{z}^2}.
\end{align*}
Notice that the left-hand side of above inequality does not change if we replace $f$ with $f - f (0)$, and
\begin{align*}
    \sup_{z \in \R^d} \frac{\abs{f (z)}}{1 + \abs{z}^2} \leq \norm{f}_\infty, \qquad
    \sup_{z \in \R^d} \frac{\abs{f (z) - f (0)}}{1 + \abs{z}^2} \leq \frac{1}{2} \oinorm{\nabla f}.
\end{align*}
Hence, it follows that
\begin{align*}
    \abs{P_{t-1} f (X_1^x) - \int_{\R^d} f (y) \mu (\dif y)}
    \leq C \eup^{-ct} (1 + \abs{X_1^x}^2)   \left( \norm{f}_\infty \land \oinorm{\nabla f}  \right),
\end{align*}
which, together with \eqref{eq:Bismut}, Lemma \ref{le:BEL} and \ref{le:Xmoment}, implies that
\begin{gather} \label{eq:le6pr3}
    \begin{split}
    &\abs{\nabla_v P_t f (x)}\\
    \leq& C \eup^{-ct} \left( \norm{f}_\infty \land \oinorm{\nabla f}  \right) \E \left[ (1 + \abs{X_1^x}^2) \abs{\int_0^1 \sca{\sigma^{-1} (X_s^x) R_s^v, \dif B_s}} \right] \\
    \leq &C \eup^{-ct}  \left( \norm{f}_\infty \land \oinorm{\nabla f}  \right) \sqrt{1 + \E \abs{X_1^x}^4} \sqrt{\E \abs{\int_0^1 \sca{\sigma^{-1} (X_s^x) R_s^v, \dif B_s}}^2} \\
    \leq& C \eup^{-ct} (1 + \abs{x}^2) \sqrt{V (x)}  \left( \norm{f}_\infty \land \oinorm{\nabla f}  \right),
    \end{split}
\end{gather}
for any $v \in \R^d$, $\abs{v} \leq 1$ and $t \geq 1$. 

Now, the proof of (i) is finished by combining \eqref{eq:le6pr3} with \eqref{eq:le6pr1} and \eqref{eq:le6pr2}.

 \noindent (ii) According to Lemma \ref{le:BEL}, for $0 < t < 1$ and any $v, w \in \R^d$, $\abs{v}, \abs{w} \leq 1$,
\begin{align*}
    \nabla_v \nabla_w P_t f (x)
    &= \nabla_v \left[ \nabla_w P_{\frac{t}{2}} \left( P_{\frac{t}{2}} f \right) \right] (x) \\
    &= \frac{2}{t} \nabla_v \E \left[ P_{\frac{t}{2}} f \left( X_{\frac{t}{2}}^x \right) \int_0^{\frac{t}{2}} \sca{\sigma^{-1} (X_s^x) R_s^w, \dif B_s} \right] \\
    &= \frac{2}{t} \E \left[ \nabla_{R_{\frac{t}{2}}^v} P_{\frac{t}{2}} f \left( X_{\frac{t}{2}}^x \right) \int_0^{\frac{t}{2}} \sca{\sigma^{-1} (X_s^x) R_s^w, \dif B_s} \right] \\
    &\quad + \frac{2}{t} \E \left[ P_{\frac{t}{2}} f \left( X_{\frac{t}{2}}^x \right) \int_0^{\frac{t}{2}} \sca{\nabla_{R_s^v} (\sigma^{-1}) (X_s^x) R_s^w, \dif B_s} \right] \\
    &\quad + \frac{2}{t} \E \left[ P_{\frac{t}{2}} f \left( X_{\frac{t}{2}}^x \right) \int_0^{\frac{t}{2}} \sca{\sigma^{-1} (X_s^x) K_s^{v,w}, \dif B_s} \right] \\
    &=: I_1 + I_2 + I_3,
\end{align*}
where $R^{w}_{s}$ and $K^{v,w}_{s}$ are defined as in \eqref{def:R,K}.

Let us prove \eqref{eq:gra21} first. For $I_{1}$, it follows from \eqref{eq:le6pr1} and the Cauchy-Schwarz inequality,
\begin{align*}
    \abs{I_1}
    &\leq \frac{C}{t \sqrt{t}} \norm{f}_\infty \E \left[ \abs{R_{\frac{t}{2}}^v} \sqrt{V \left( X_{\frac{t}{2}}^x \right)} \abs{\int_0^{\frac{t}{2}} \sca{\sigma^{-1} (X_s^x) R_s^w, \dif B_s}} \right] \\
    &\leq \frac{C}{t \sqrt{t}} \norm{f}_\infty \sqrt{\E \left[ \abs{R_{\frac{t}{2}}^v}^2 V \left( X_{\frac{t}{2}}^x \right) \right]} \sqrt{\E \abs{\int_0^{\frac{t}{2}} \sca{\sigma^{-1} (X_s^x) R_s^w, \dif B_s}}^2} \\
    &\leq \frac{C}{t} V (x) \norm{f}_\infty.
\end{align*}
For $I_{2}$ and $I_{3}$, by \eqref{ineq:Rmoment1} and \eqref{ineq:Rmoment2}, we have
\begin{gather*}
    \abs{I_2}
    \leq \frac{2}{t} \norm{P_{\frac{t}{2}} f}_\infty \sqrt{\E \abs{\int_0^\frac{t}{2} \sca{\nabla_{R_s^v} (\sigma^{-1}) (X_s^x) R_s^w, \dif B_s}}^2}
    \leq \frac{C}{\sqrt{t}} \sqrt{V (x)} \norm{f}_\infty, \\
    \abs{I_3}
    \leq \frac{2}{t} \norm{P_{\frac{t}{2}} f}_\infty \sqrt{\E \abs{\int_0^\frac{t}{2} \sca{\sigma^{-1} (X_s^x) K_s^{v,w}, \dif B_s}}^2}
    \leq \frac{C}{\sqrt{t}} \sqrt{V (x)} \norm{f}_\infty.
\end{gather*}
Combining above estimates of $I_1$, $I_2$, and $I_3$ derives
\begin{align} \label{eq:le6pr4}
    \abs{\nabla_v \nabla_w P_t f (x)}
    \leq \frac{C}{t} V (x) \norm{f}_\infty.
\end{align}
\par
 Now, let us prove \eqref{eq:gra22} for $0 < t < 1$. For $I_{1}$, it follows from \eqref{eq:le6pr2}, the Cauchy-Schwarz inequality and the inequality $1+|x|^{r+1}\le CV(x)$ for some constant $C$ that  
\begin{align*}
    \abs{I_{1}} 
    &\le \frac{2}{t} \E \left[\abs{ \nabla_{R_{\frac{t}{2}}^v} P_{\frac{t}{2}} f \left( X_{\frac{t}{2}}^x \right)} \abs{\int_0^{\frac{t}{2}} \sca{\sigma^{-1} (X_s^x) R_s^w, \dif B_s}} \right] \\
    &\le  \frac{C}{t} \oinorm{\nabla f} \E \left[\left(1+\abs{X_{\frac{t}{2}}^x}^{r+1} \right) \sqrt{V(X_{\frac{t}{2}}^x)} \abs{\int_0^{\frac{t}{2}} \sca{\sigma^{-1} (X_s^x) R_s^w, \dif B_s}} \right]\\
    &\le \frac{C}{t} \oinorm{\nabla f} \left(\E  \left(1+\abs{X_{\frac{t}{2}}^x}^{r+1} \right)^{4}  \right)^{\frac{1}{4}} \left( \E\abs{ V(X_{\frac{t}{2}}^x)}^{2}\right)^{\frac{1}{4}} \\
    &\quad \times\left(\E  \abs{\int_0^{\frac{t}{2}} \sca{\sigma^{-1} (X_s^x) R_s^w, \dif B_s}}^{2} \right)^{\frac{1}{2}}\\
    &\le \frac{C}{\sqrt{t}} V (x)^{\frac{3}{2}} \oinorm{\nabla f}.
\end{align*}
For $I_{2}$ and $I_{3}$, it follows from the Cauchy-Schwarz inequality, and \eqref{eq:le6pr2} that
\begin{align*}
    \abs{I_2} + \abs{I_3} 
    \leq& \frac{2}{t} \E \left[ \abs{P_{\frac{t}{2}} f \left( X_{\frac{t}{2}}^x \right) -P_{\frac{t}{2}}f(x)} \abs{\int_0^{\frac{t}{2}} \sca{\nabla_{R_s^v} (\sigma^{-1}) (X_s^x) R_s^w, \dif B_s}} \right] \\
    &  + \frac{2}{t} \E \left[  \abs{P_{\frac{t}{2}} f \left( X_{\frac{t}{2}}^x \right) -P_{\frac{t}{2}}f(x)} \abs{\int_0^{\frac{t}{2}} \sca{\sigma^{-1} (X_s^x) K_s^{v,w}, \dif B_s}} \right]\\
    \le &\frac{2}{t}\E \left[ \left(\int_0^1 \opnorm{\nabla P_{\frac{t}{2}} f((1-r)X_{\frac{t}{2}}^x+rx)} \dif r \right)\abs{ X_{\frac{t}{2}}^x-x } \abs{\int_0^{\frac{t}{2}} \sca{\nabla_{R_s^v} (\sigma^{-1}) (X_s^x) R_s^w, \dif B_s}} \right] \\
    &  + \frac{2}{t} \E \left[\left(\int_0^1 \opnorm{\nabla P_{\frac{t}{2}} f((1-r)X_{\frac{t}{2}}^x+rx)} \dif r\right) \abs{ X_{\frac{t}{2}}^x-x } \abs{\int_0^{\frac{t}{2}} \sca{\sigma^{-1} (X_s^x) K_s^{v,w}, \dif B_s}} \right]\\
    \le &\frac{2}{t} \left( \E\left(\int_0^1 \opnorm{\nabla P_{\frac{t}{2}} f((1-r)X_{\frac{t}{2}}^x+rx)} \dif r \right)^4\right)^{\frac14} \left(\E\abs{ X_{\frac{t}{2}}^x-x }^4  \right)^{\frac14} \\
    & \cdot\left[\left(\E \abs{\int_0^{\frac{t}{2}} \sca{\nabla_{R_s^v} (\sigma^{-1}) (X_s^x) R_s^w, \dif B_s}}^2 \right)^{\frac12}+\left(\E \abs{\int_0^{\frac{t}{2}} \sca{\sigma^{-1} (X_s^x) K_s^{v,w}, \dif B_s}}^2 \right)^{\frac12}\right]\\
     \leq& C (1 + \abs{x}^{r+1}) V (x) \oinorm{\nabla f},
\end{align*}
where the last inequality comes from Lemma \ref{le:Xmoment}, \ref{le:onestep} and \ref{le:BEL1}.

Combining above estimates of $I_1$, $I_2$, and $I_3$ derives
\begin{align} \label{eq:le6pr5}
    \abs{\nabla_v \nabla_w P_t f (x)}
    \leq \frac{C}{\sqrt{t}} V (x)^{\frac{3}{2}} \oinorm{\nabla f}.
\end{align}

Then we turn to the case $t \geq 1$. We still have
\begin{align*}
    \nabla_v \nabla_w P_t f (x)
    &= 2 \E \left[ \nabla_{R_{\frac{1}{2}}^v} P_{t - \frac{1}{2}} f \left( X_{\frac{1}{2}}^x \right) \int_0^{\frac{1}{2}} \sca{\sigma^{-1} (X_s^x) R_s^w, \dif B_s} \right] \\
    &\quad + 2 \E \left[ \left( P_{t - \frac{1}{2}} f \left( X_{\frac{1}{2}}^x \right) - \int_{\R^d} f (y) \mu (\dif y) \right) \int_0^{\frac{1}{2}} \sca{\nabla_{R_s^v} (\sigma^{-1}) (X_s^x) R_s^w, \dif B_s} \right] \\
    &\quad + 2 \E \left[ \left( P_{t - \frac{1}{2}} f \left( X_{\frac{1}{2}}^x \right) - \int_{\R^d} f (y) \mu (\dif y) \right) \int_0^{\frac{1}{2}} \sca{\sigma^{-1} (X_s^x) K_s^{v,w}, \dif B_s} \right] \\
    &=: I_4 + I_5 + I_6.
\end{align*}
By \eqref{eq:gra1_tot1} and \eqref{eq:gra1_tot2}, we have
\begin{align*}
    \abs{\nabla_v P_{t - \frac{1}{2}} f \left( X_{\frac{1}{2}}^x \right)}
    \leq C \eup^{-ct} |v|V \left( X_{\frac{1}{2}}^x \right) \left( \norm{f}_\infty \land \oinorm{\nabla f} \right),
\end{align*}
and according to Lemma \ref{ergodic} and \cite[Theorem 2.5 (a)]{BB2006}, it can be shown that
\begin{align*}
    \abs{P_{t - \frac{1}{2}} f \left( X_{\frac{1}{2}}^x \right) - \int_{\R^d} f (y) \mu (\dif y)}
    \leq C \eup^{-ct} \left( 1 + \abs{X_{\frac{1}{2}}^x}^2 \right) \left( \norm{f}_\infty \land \oinorm{\nabla f} \right),
\end{align*}
which implies
\begin{align*}
    \abs{I_4}
    &\leq C \eup^{-ct} V (x)^{\frac{3}{2}} \left( \norm{f}_\infty \land \oinorm{\nabla f} \right), \\
    \abs{I_5}
    &\leq C \eup^{-ct} (1 + \abs{x}^2) \sqrt{V (x)} \left( \norm{f}_\infty \land \oinorm{\nabla f} \right), \\
    \abs{I_6}
    &\leq C \eup^{-ct} (1 + \abs{x}^2) V (x) \left( \norm{f}_\infty \land \oinorm{\nabla f} \right).
\end{align*}
So we get
\begin{align} \label{eq:le6pr6}
    \abs{\nabla_v \nabla_w P_t f (x)}
    \leq C \eup^{-ct} V (x)^{\frac{3}{2}} \left( \norm{f}_\infty \land \oinorm{\nabla f} \right),
\end{align}
for any $v, w \in \R^d$, $\abs{v}, \abs{w} \leq 1$ and $t \geq 1$. 

The desired result follows from \eqref{eq:le6pr4}, \eqref{eq:le6pr5}, and \eqref{eq:le6pr6}.
\end{proof}

\section{Proof of Main Results}{\label{proof}}
In main theorems of this article, i.e. Theorem \ref{thm1} and \ref{thm2}, our goal is to prove that for any $\alpha \in (0, 1 / 2)$, there exists the constant $C  > 0$ such that,
\begin{align*}
        \W_1 (\LL (X_{t_n}), \LL (Y_{t_n})) &\leq C \eta_n^\alpha, \quad \forall n \geq 1, \\
        \dtv (\LL (X_{t_n}), \LL (Y_{t_n})) &\leq C \eta_n^\alpha, \quad \forall n \geq 1.
    \end{align*}

By the Kantorovich-Rubinstein theorem \cite{Villani2003} and a standard approximation method, it is sufficient to show that,
\begin{align*}
    \abs{\E f(X_{t_n}) - \E f(Y_{t_n})} \leq C \eta_n^\alpha \left( \norm{f}_\infty \land \oinorm{\nabla f} \right), \quad \forall n \geq 1, \; f \in \mathcal{C}_{b}^2 (\R^d).
\end{align*}
For fixed $n \geq 1$ and $f \in \mathcal{C}_{b}^2 (\R^d)$, by the domino decomposition, we have 
\begin{align} \label{eq:thm1pr1}
    \begin{split}
    \E f(X_{t_n}) - \E f(Y_{t_n})
    &= P_{0, t_n} f (x_0) - Q_{0, t_n} f (x_0) \\
    &= \sum_{k = 1}^{n} Q_{0, t_{k-1}} (P_{t_{k-1}, t_k} - Q_{t_{k-1}, t_k}) P_{t_k, t_n} f (x_0) \\
    &= \sum_{k = 1}^{n} \E \left[ (P_{t_{k-1}, t_k} - Q_{t_{k-1}, t_k}) P_{t_k, t_n} f (Y_{t_{k-1}}) \right].
    \end{split}
\end{align}

Based on \eqref{eq:thm1pr1}, we provide an estimate for the final step (i.e., $\lvert (P_{t_{n-1}, t_n} - Q_{t_{n-1}, t_n}) f (x) \rvert$) first, which shows in Lemma \ref{le:laststep}, and then provide the complete proof.
\subsection{The estimate of the last step}
\begin{lemma} \label{le:laststep}
    Suppose Assumption \ref{A1} and \ref{A2} hold. There exists a constant $C > 0$ such that for any $x\in \R^{d}$, $n\ge 1$ and $f \in \mathcal{C}_b^2 (\R^d)$
    \begin{align*}
        \abs{(P_{t_{n-1}, t_n} - Q_{t_{n-1}, t_n}) f (x)}
        \leq C \sqrt{\eta_n}(1+|x|^{2r+1})V(x)  \norm{f}_\infty.
    \end{align*}
     where $V(x)$ is  a smooth function defined in \eqref{Lyapu}. 
\end{lemma}

\begin{proof}
    Let $\{\Tilde{Q}_{t}\}_{\{t\ge 0\}}$ be the semigroup defined by
    \begin{align*}
        \Tilde{Q}_{t}f(x):=\E \left[f(\Tilde{Y}_{t}^{x}) \right], \quad \forall f\in \mathcal{C}_b^2 (\R^d),
    \end{align*}
    where $\Tilde{Y}_{t}^{x}$ is the stochastic process given by the following time-homogeneous SDE
    \begin{align*}
        \dif \Tilde{Y}_{t}^{x}=\frac{b (x)}{1 + \eta_{n}^\alpha \opnorm{\nabla b (x)}} \, \dif t+\sigma(x) \, \dif B_t, \quad \Tilde{Y}_{0}^{x}=x.
    \end{align*}
    The desired result is equivalent to 
    \begin{align*}
        \abs{\left(P_{\eta_{n}} -\Tilde{Q}_{\eta_{n}}\right)f(x)}\le C \sqrt{\eta_n}(1+|x|^{2r+1})V(x) \norm{f}_\infty, \quad \forall x\in \R^{d}, \; f\in \mathcal{C}_b^2 (\R^d).
    \end{align*}
    
    By the Duhamel's principle, for any $t\ge 0$,
    \begin{equation} \label{eq:Duhamel}
        \begin{aligned}
            P_{t}f(x)-\Tilde{Q}_{t}f(x)
        &=\int_{0}^{t}  \frac{\dif}{\dif s}\Tilde Q_{t-s}(P_{s}f)(x) \, \dif s\\
        &=\int_{0}^{t} \Tilde Q_{t-s}\left(\mathcal{A}^{P}-\mathcal{A}^{\Tilde{Q}} \right)(P_{s}f)(x) \, \dif s\\
        &=\int_{0}^{t}  \E\left(\mathcal{A}^{P}-\mathcal{A}^{\Tilde{Q}} \right)(P_{s}f)(\Tilde{Y}_{t-s}^{x}) \, \dif s,
        \end{aligned}
    \end{equation}
    with $\mathcal{A}^{P}$ and $\mathcal{A}^{\Tilde{Q}}$ being the corresponding infinitesimal generator of $P_{t}$ and $\Tilde{Q}_{t}$, i.e., for any $h\in \mathcal{C}_b^2 (\R^d)$,
    \begin{align*}
        \mathcal{A}^{P}h(\cdot)&:=\lim_{t\downarrow 0}\frac{P_{t}h(\cdot)-h(\cdot)}{t}=\langle \nabla h(\cdot),b(\cdot) \rangle+\frac{1}{2}\langle \nabla^{2}h(\cdot), \sigma(\cdot)\sigma(\cdot)^{T} \rangle_{\mathrm{HS}}, \\
        \mathcal{A}^{\Tilde{Q}}h(\cdot) &:=\lim_{t\downarrow 0}\frac{\Tilde{Q}_{t}h(\cdot)-h(\cdot)}{t}=\left\langle \nabla h(\cdot), \frac{b(x)}{1+\eta_{n}^{\alpha}\opnorm{\nabla b(x)}} \right\rangle+\frac{1}{2}\langle \nabla^{2}h(\cdot), \sigma(x)\sigma(x)^{T} \rangle_{\mathrm{HS}}.
    \end{align*}
    
    We now provide the estimate of $\left|\E(\mathcal{A}^{P}-\mathcal{A}^{\Tilde{Q}} )(P_{s}f)(\Tilde{Y}_{t-s}^{x})\right|$. It follows from Lemma \ref{le:gradient}, \ref{le:onestep}, \ref{le:Ymoment} and \eqref{eq:drift} that, for any $s<t\le \eta_{n}$,
    \begin{align} \label{eq:generator1}
    \begin{split}
    &\mathrel{\phantom{=}} \abs{\E\left\langle \nabla P_{s}f(\Tilde{Y}_{t-s}^{x}),b(\Tilde{Y}_{t-s}^{x})- \frac{b(x)}{1+\eta_{n}^{\alpha}\opnorm{\nabla b(x)}} \right\rangle}\\
    &\le C\sqrt{\E \opnorm{\nabla P_{s}f(\Tilde{Y}_{t-s}^{x})}^{2} }\sqrt{ \E\left[\left( 1+|x|^{2r} + \lvert \Tilde{Y}_{t-s}^{x} \rvert^{2} \right) \lvert \Tilde{Y}_{t-s}^{x}-x \rvert^{2}\right]+ \eta_{n}^{2\alpha}(1+|x|^{2r+1})^{2}}\\
    &\le C\frac{\norm{f}_{\infty}}{\sqrt{s}}(1+|x|^{2r+1})\sqrt{\E [V(\Tilde{Y}_{t-s}^{x})^{2}]} \sqrt{\eta_{n}+\eta_{n}^{2\alpha}}\\
        &\le C\eta_{n}^{\alpha} (1+|x|^{2r+1})V(x)\frac{\|f\|_{\infty}}{\sqrt{s}}.
    \end{split}
    \end{align}
    What's more, notice that the distribution of $\Tilde{Y}_{t-s}^{x}$ is 
    \begin{align*}
        \Tilde{Y}_{t-s}^{x}\sim \mathcal{N}\left(\Tilde{\mu}_{t-s}, \Tilde{\Sigma}_{t-s}\right),
    \end{align*}
    with
    \begin{align*}
        \Tilde{\mu}_{t-s}:=x+\frac{(t-s)b(x)}{1 + \eta_{n}^\alpha \opnorm{\nabla b (x)}},\quad \text{and} \quad \Tilde{\Sigma}_{t-s}:=(t-s)\sigma(x)\sigma(x)^{T},
    \end{align*}
    and denote its probability density function by $\Tilde{p}_{t-s}$. It can be easily verified that
    \begin{align*}
        \nabla \Tilde{p}_{t-s}(y)=-\Tilde{\Sigma}_{t-s}^{-1} (y-\Tilde{\mu}_{t-s})\Tilde{p}_{t-s}(y).
    \end{align*}
    So it follows from the integration by part formula, Cauchy-Schwarz inequality, Assumption \ref{A2} and Lemma \ref{le:gradient} that for any $0< s\le t\le \eta_{n}$,
    \begin{align} \label{eq:generator2}
        \begin{split}
            &\mathrel{\phantom{=}} \abs{\E\left\langle \nabla^{2}P_{s}f(\Tilde{Y}_{t-s}^{x}), \sigma(\Tilde{Y}_{t-s}^{x})\sigma(\Tilde{Y}_{t-s}^{x})^{T}-\sigma(x)\sigma(x)^{T}  \right\rangle_{\mathrm{HS}} }\\
            &=\abs{\int_{\R^{d}} \left\langle \nabla^{2}P_{s}f(y), \sigma(y)\sigma(y)^{T}-\sigma(x)\sigma(x)^{T}  \right\rangle_{\mathrm{HS}}  \Tilde{p}_{t-s}(y) \, \dif y}\\
            &\le \abs{\int_{\R^{d}} \sum_{i,j=1}^{d} \partial_{i}P_{s}f(y) \left[\sigma(y)\sigma(y)^{T}-\sigma(x)\sigma(x)^{T}\right]_{ij}  \partial_{j}\Tilde{p}_{t-s}(y) \, \dif y }\\
            &\quad + \abs{\int_{\R^{d}} \sum_{i,j=1}^{d} \partial_{i}P_{s}f(y) \partial_{j}\left[\sigma(y)\sigma(y)^{T}-\sigma(x)\sigma(x)^{T}\right]_{ij}  \Tilde{p}_{t-s}(y) \, \dif y }\\
            &\le \int_{\R^{d}}  \abs{\nabla P_{s}f(y)} \abs{\nabla \Tilde{p}_{t-s}(y)} \oinorm{\sigma(y)\sigma(y)^{T}-\sigma(x)\sigma(x)^{T}} \dif y\\
            &\quad +\int_{\R^{d}}\abs{\nabla P_{s}f(y)} \sqrt{\sum_{i=1}^{d}\left( \sum_{j=1}^{d} \partial_{j}\left[\sigma(y)\sigma(y)^{T}\right]_{ij}\right)^{2} } \Tilde{p}_{t-s}(y) \, \dif y\\
            &\le C\frac{\|f\|_{\infty}}{\sqrt{s}}\int_{\R^{d}}  V(y)\left(\frac{|y-\Tilde{\mu}_{t-s}|\abs{y-x}}{t-s}+1 \right) \Tilde{p}_{t-s}(y) \, \dif y\\
            &\le C(1+|x|^{2r+1})V(x)\frac{\|f\|_{\infty}}{\sqrt{s}},
        \end{split}
    \end{align}
    where $[A]_{ij}$ denotes the $i$-th row and $j$-th column element of matrix $A$.
    
    Now, combining \eqref{eq:generator1} and \eqref{eq:generator2} together with \eqref{eq:Duhamel} gives us
    \begin{align*}
        \abs{P_{\eta_{n}}f(x)-\Tilde{Q}_{\eta_{n}}f(x)}
        \le C\|f\|_{\infty}\int_{0}^{\eta_{n}}  \frac{\eta_{n}^{\alpha}+1}{\sqrt{s}} \, \dif s
        \le C\sqrt{\eta_{n}}(1+|x|^{2r+1})V(x)\|f\|_{\infty},
    \end{align*}
    and the desired result follows.
\end{proof}

\subsection{Proof of main results}
Before providing the proof of main results, we first state the following technical lemma, which will be proved in Appendix \ref{appendix}.
 
\begin{lemma} \label{le:A1}
    For any $\beta \in (0, 1 / 2]$ and $c > 0$, there exists a constant $C > 0$ such that, if Assumption \ref{A3} holds with $\eta_1 < 1$ and $\theta < c \eup^{-c} / \beta$, we have
    \begin{align*}
        \sum_{k = 1}^n \eta_k^{1 + \beta} \eup^{-c (t_n - t_k)}
        \leq C \eta_n^\beta, \qquad
        \sum_{k = K_n}^{n-1} \frac{\eta_k^{1 + \beta}}{\sqrt{t_n - t_k}}
        \leq C \eta_n^\beta, \qquad
        \sum_{k = K_n}^{n-1} \frac{\eta_k^{1 + \beta}}{t_n - t_k}
        \leq C \eta_n^\beta \abs{\ln \eta_n},
    \end{align*}
    where $t_k = \sum_{i = 1}^k \eta_i$, $K_n := \min \{ k \geq 1 \colon t_n - t_k \leq 1 \}$, and $C$ depends on $\beta$, $c$, $\eta_1$, and $\theta$.
\end{lemma}

Now, we present the proofs of the main theorems of this paper,
\begin{proof}[Proof of Theorem \ref{thm1}]
To reach the desired result, we only need to prove
\begin{align*}
    \abs{\E f(X_{t_n}) - \E f(Y_{t_n})} \leq C \eta_n^\alpha \left( \norm{f}_\infty \land \oinorm{\nabla f} \right), \quad \forall n \geq 1, \; f \in \mathcal{C}_{b}^2 (\R^d).
\end{align*}
By \eqref{eq:thm1pr1}, for fixed $n \geq 1$ and $f \in \mathcal{C}_{b}^2 (\R^d)$, we have 
\begin{align*} 
    \E f(X_{t_n}) - \E f(Y_{t_n})=\sum_{k = 1}^{n} \E \left[ (P_{t_{k-1}, t_k} - Q_{t_{k-1}, t_k}) P_{t_k, t_n} f (Y_{t_{k-1}}) \right].
\end{align*}

For $k = 1, \dots, n-1$ and $g \in \mathcal{C}_{b}^2 (\R^d)$, notice that
\begin{align} \label{eq:thm1pr2}
    \begin{split}
    &\mathrel{\phantom{=}} (P_{t_{k-1}, t_k} - Q_{t_{k-1}, t_k}) g (x) \\
    &= \E \left[ g (X_{t_{k-1}, t_k}^x) - g (Y_{t_{k-1}, t_k}^x) \right] \\
    &= \E \sca{\nabla g (x), X_{t_{k-1}, t_k}^x - Y_{t_{k-1}, t_k}^x} \\
    &\quad + \E \int_0^1 \sca{\nabla g (r X_{t_{k-1}, t_k}^x + (1-r) Y_{t_{k-1}, t_k}^x) - \nabla g (x), X_{t_{k-1}, t_k}^x - Y_{t_{k-1}, t_k}^x} \dif r \\
    &= \sca{\nabla g (x), \E \Delta_{t_k}^x} + \int_0^1 \dif r \int_0^1 \E \left[ \nabla_{(\Xi_{t_k}^{x,r} - x)} \nabla_{\Delta_{t_k}^x} g ( s \Xi_{t_k}^{x,r} + (1-s) x ) \right] \dif s,
    \end{split}
\end{align}
where $\Delta_{t_k}^x := X_{t_{k-1}, t_k}^x - Y_{t_{k-1}, t_k}^x$, $\Xi_{t_k}^{x,r} := r X_{t_{k-1}, t_k}^x + (1-r) Y_{t_{k-1}, t_k}^x$, and $\{ X_{t_{k-1}, t}^x \}_{t \in [t_{k-1}, t_k]}$ and $\{ Y_{t_{k-1}, t}^x \}_{t \in [t_{k-1}, t_k]}$ satisfy
\begin{align*}
    \dif X_{t_{k-1}, t}^x &= b (X_{t_{k-1}, t}^x) \, \dif t + \sigma (X_{t_{k-1}, t}^x) \, \dif B_t, & X_{t_{k-1}, t_{k-1}}^x &= x, \\
    \dif Y_{t_{k-1}, t}^x &= \frac{b (x)}{1 + \eta_k^\alpha \opnorm{\nabla b (x)}} \, \dif t + \sigma (x) \, \dif B_t, & Y_{t_{k-1}, t_{k-1}}^x &= x.
\end{align*}

Combining Lemma \ref{le:Xmoment} and \ref{le:onestep}, we have the following estimate of $\E \abs{\Delta_{t_k}^x}^4$, $\E \abs{\Xi_{t_k}^{x,r} - x}^4$ and $\abs{\E \Delta_{t_k}^x}$, i.e.
\begin{align*}
    \E \abs{\Delta_{t_k}^x}^4
    &\leq C \eta_k^4 (1 + \abs{x}^{2r+1})^4, \\
    \E \abs{\Xi_{t_k}^{x,r} - x}^4
    &\leq C \left( \E \abs{X_{t_{k-1}, t_k}^x - x}^4 + \E \abs{Y_{t_{k-1}, t_k}^x - x}^4 \right)
    \leq C \eta_k^2 (1 + \abs{x}^{r+1})^4,
\intertext{and}
    \abs{\E \Delta_{t_k}^x}
    &= \abs{\E \int_{t_{k-1}}^{t_k} \frac{b (X_{t_{k-1}, t}^x) - b (x) + \eta_k^\alpha \opnorm{\nabla b (x)} b (X_{t_{k-1}, t}^x)}{1 + \eta_k^\alpha \opnorm{\nabla b (x)}} \, \dif t} \\
    &\leq \int_{t_{k-1}}^{t_k} \left( \E \abs{b (X_{t_{k-1}, t}^x) - b (x)} + \eta_k^\alpha \opnorm{\nabla b (x)} \E \abs{b (X_{t_{k-1}, t}^x)} \right) \dif t \\
    &\leq C \eta_k^{1 + \alpha} (1 + \abs{x}^{2r+1}).
\end{align*}
Together with Lemma \ref{le:gradient}, taking $g = P_{t_k, t_n} f$ in \eqref{eq:thm1pr2} derives that
\begin{align*}
    &\mathrel{\phantom{=}} \abs{(P_{t_{k-1}, t_k} - Q_{t_{k-1}, t_k}) P_{t_k, t_n} f (x)} \\
    &\leq \opnorm{\nabla P_{t_k, t_n} f (x)} \abs{\E \Delta_{t_k}^x} \\
    &\quad + \int_0^1 \dif r \int_0^1 \left( \E \abs{\Xi_{t_k}^{x,r} - x}^4 \right)^{\frac{1}{4}} \left( \E \abs{\Delta_{t_k}^x}^4 \right)^{\frac{1}{4}} \left( \E \opnorm{\nabla^2 P_{t_k, t_n} f (s \Xi_{t_k}^{x,r} + (1-s) x)}^2 \right)^{\frac{1}{2}} \dif s \\
    &\leq \left( \norm{f}_\infty \land \oinorm{\nabla f} \right) \left[ \frac{C \eta_k^{1 + \alpha} \eup^{-c (t_n - t_k)}}{\sqrt{(t_n - t_k) \land 1}} (1 + \abs{x}^{2r+1}) V (x) \right. \\
    &\qquad \qquad \qquad \left. + \frac{C \eta_k^{\frac{3}{2}} \eup^{-c (t_n - t_k)}}{(t_n - t_k) \land 1} (1 + \abs{x}^{3r+2}) \int_0^1 \dif r \int_0^1 \sqrt{\E \left[ V (s \Xi_{t_k}^{x,r} + (1-s) x)^3 \right]} \, \dif s \right].
\end{align*}
For $0 \leq r, s \leq 1$, Hölder's inequality and Lemma \ref{le:Xmoment}, \ref{le:Ymoment} imply
\begin{align*}
    \E \left[ V (s \Xi_{t_k}^{x,r} + (1-s) x)^3 \right]
    &\leq C V (x)^{3 (1-s)} \left\{ \E \left[ V (X_{t_{k-1}, t_k}^x)^3 \right] \right\}^{sr} \left\{ \E \left[ V (Y_{t_{k-1}, t_k}^x)^3 \right] \right\}^{s (1-r)} \\
    &\leq C V (x)^3,
\end{align*}
so we have, for $k = 1, \dots, n-1$, 
\begin{align*}
    &\mathrel{\phantom{=}} \abs{(P_{t_{k-1}, t_k} - Q_{t_{k-1}, t_k}) P_{t_k, t_n} f (x)} \\
    &\leq \left[ \frac{C \eta_k^{1 + \alpha} \eup^{-c (t_n - t_k)}}{\sqrt{(t_n - t_k) \land 1}} + \frac{C \eta_k^{\frac{3}{2}} \eup^{-c (t_n - t_k)}}{(t_n - t_k) \land 1} \right] V (x)^2 \left( \norm{f}_\infty \land \oinorm{\nabla f} \right),
\end{align*}
Since Lemma \ref{le:Ymoment} shows $\sup_{k \geq 1} \E [ V (Y_{t_{k-1}})^2 ] < + \infty$, it follows from Assumption \ref{A3} and Lemma \ref{le:A1} that
\begin{align} \label{eq:thm1pr3}
    \begin{split}
    &\mathrel{\phantom{=}} \sum_{k = 1}^{n-1} \E \abs{(P_{t_{k-1}, t_k} - Q_{t_{k-1}, t_k}) P_{t_k, t_n} f (Y_{t_{k-1}})} \\
    &\leq C \left( \norm{f}_\infty \land \oinorm{\nabla f} \right) \sum_{k = 1}^{n-1} \left[ \frac{\eta_k^{1 + \alpha} \eup^{-c (t_n - t_k)}}{\sqrt{(t_n - t_k) \land 1}} + \frac{\eta_k^{\frac{3}{2}} \eup^{-c (t_n - t_k)}}{(t_n - t_k) \land 1} \right] \E \left[ V (Y_{t_{k-1}})^2 \right] \\
    &\leq C \left( \eta_n^\alpha + \eta_n^{\frac{1}{2}} \abs{\ln \eta_n} \right) \left( \norm{f}_\infty \land \oinorm{\nabla f} \right) \\
    &\leq C \eta_n^\alpha \left( \norm{f}_\infty \land \oinorm{\nabla f} \right).
    \end{split}
\end{align}

For $k = n$, Lemma \ref{le:onestep} shows
\begin{align*}
    \abs{(P_{t_{n-1}, t_n} - Q_{t_{n-1}, t_n}) f (x)}
    \leq \oinorm{\nabla f} \E \abs{X_{t_{n-1}, t_n}^x - Y_{t_{n-1}, t_n}^x}
    \leq C \eta_n (1 + \abs{x}^{2r+1}) \oinorm{\nabla f}.
\end{align*}
Together with Lemma \ref{le:laststep} and \ref{le:Ymoment}, we have
\begin{align} \label{eq:thm1pr4}
\begin{split}
    &\E \abs{(P_{t_{n-1}, t_n} - Q_{t_{n-1}, t_n}) f (Y_{t_{n-1}})}\\
    \leq& C \eta_n^\alpha \E\left[(1+|Y_{t_{n-1}}|^{2r+1})V(Y_{t_{n-1}})\right]\left( \norm{f}_\infty \land \oinorm{\nabla f} \right)\\
    \leq& C \eta_n^\alpha \left( \norm{f}_\infty \land \oinorm{\nabla f} \right).
\end{split}
\end{align}
where the second inequality comes from the fact that $|x|^{2r+1}\le C e^{|x|}+1$.

Combining \eqref{eq:thm1pr1}, \eqref{eq:thm1pr3}, and \eqref{eq:thm1pr4}, we have
\begin{align}
|\E f(X_{t_n}) - \E f(Y_{t_n})|\leq C \eta_n^\alpha \left( \norm{f}_\infty \land \oinorm{\nabla f} \right).
\end{align}
so we have proved the desired result.
\end{proof}

\begin{proof}[Proof of Theorem \ref{thm2}]
For $k = 1, \dots, n-1$, we have
\begin{align*}
    &\mathrel{\phantom{=}} \abs{(P_{t_{k-1}, t_k} - Q_{t_{k-1}, t_k}) P_{t_k, t_n} f (x)} \\
    &= \abs{\E \left[ P_{t_k, t_n} f (X_{t_{k-1}, t_k}^x) - P_{t_k, t_n} f (Y_{t_{k-1}, t_k}^x) \right]} \\
    &= \abs{\int_0^1 \E \sca{\nabla P_{t_k, t_n} f (r X_{t_{k-1}, t_k}^x + (1-r) Y_{t_{k-1}, t_k}^x), X_{t_{k-1}, t_k}^x - Y_{t_{k-1}, t_k}^x} \dif r} \\
    &\leq \int_0^1 \sqrt{\E \opnorm{\nabla P_{t_k, t_n} f (r X_{t_{k-1}, t_k}^x + (1-r) Y_{t_{k-1}, t_k}^x)}^2 \E \abs{X_{t_{k-1}, t_k}^x - Y_{t_{k-1}, t_k}^x}^2} \, \dif r.
\end{align*}
Since $\sigma \equiv \sigma_0 \in \R^{d \times d}$, Lemma \ref{le:onestep} and \ref{le:gradient} show that
\begin{gather*}
    \E \abs{X_{t_{k-1}, t_k}^x - Y_{t_{k-1}, t_k}^x}^2
    \leq C \eta_k^{2 + 2 \alpha} ( 1 + \abs{x}^{4r+2} ), \\
    \opnorm{\nabla P_{t_k, t_n} f (\Xi_{t_k}^{x,r})}
    \leq \frac{C \eup^{-c (t_n - t_k)}}{\sqrt{(t_n - t_k) \land 1}} V (\Xi_{t_k}^{x,r}) \left( \norm{f}_\infty \land \oinorm{\nabla f} \right),
\end{gather*}
where $\Xi_{t_k}^{x,r} = r X_{t_{k-1}, t_k}^x + (1-r) Y_{t_{k-1}, t_k}^x$. So we have
\begin{align*}
    &\mathrel{\phantom{=}} \abs{(P_{t_{k-1}, t_k} - Q_{t_{k-1}, t_k}) P_{t_k, t_n} f (x)} \\
    &\leq \frac{C \eta_k^{1 + \alpha} \eup^{-c (t_n - t_k)}}{\sqrt{(t_n - t_k) \land 1}} ( 1 + \abs{x}^{r+1} ) \left( \norm{f}_\infty \land \oinorm{\nabla f} \right) \int_0^1 \sqrt{\E \left[ V (\Xi_{t_k}^{x,r})^2 \right]} \, \dif r \\
    &\leq \frac{C \eta_k^{1 + \alpha} \eup^{-c (t_n - t_k)}}{\sqrt{(t_n - t_k) \land 1}} V (x)^2 \left( \norm{f}_\infty \land \oinorm{\nabla f} \right),
\end{align*}
where in the last inequality we use the estimates in Lemma \ref{le:Xmoment} and \ref{le:Ymoment}. 

Since Lemma \ref{le:Ymoment} shows $\sup_{k \geq 1} \E [ V (Y_{t_{k-1}})^2 ] < + \infty$, it follows from Lemma \ref{le:A1} that
\begin{align} \label{eq:thm2pr2}
    \begin{split}
    &\mathrel{\phantom{=}} \sum_{k = 1}^{n-1} \E \abs{(P_{t_{k-1}, t_k} - Q_{t_{k-1}, t_k}) P_{t_k, t_n} f (Y_{t_{k-1}})} \\
    &\leq C \left( \norm{f}_\infty \land \oinorm{\nabla f} \right) \sum_{k = 1}^{n-1} \frac{\eta_k^{1 + \alpha} \eup^{-c (t_n - t_k)}}{\sqrt{(t_n - t_k) \land 1}} \E \left[ V (Y_{t_{k-1}})^2 \right] \\
    &\leq C \eta_n^\alpha \left( \norm{f}_\infty \land \oinorm{\nabla f} \right).
    \end{split}
\end{align}

For $k = n$, Lemma \ref{le:onestep} shows
\begin{align*}
    \abs{(P_{t_{n-1}, t_n} - Q_{t_{n-1}, t_n}) f (x)}
    \leq \oinorm{\nabla f} \E \abs{X_{t_{n-1}, t_n}^x - Y_{t_{n-1}, t_n}^x}
    \leq C \eta_n (1 + \abs{x}^{2r+1}) \oinorm{\nabla f}.
\end{align*}
Together with Lemma \ref{le:laststep} and \ref{le:Ymoment}, we have
\begin{align} \label{eq:thm2pr3}
\begin{split}
    &\E \abs{(P_{t_{n-1}, t_n} - Q_{t_{n-1}, t_n}) f (Y_{t_{n-1}})}\\
    \leq &C \eta_n^\alpha \E\left[(1+|Y_{t_{n-1}}|^{2r+1})V(Y_{t_{n-1}})\right]\left( \norm{f}_\infty \land \oinorm{\nabla f} \right)\\
    \leq& C \eta_n^\alpha \left( \norm{f}_\infty \land \oinorm{\nabla f} \right).
\end{split}
\end{align}
The desired result follows from \eqref{eq:thm1pr1}, \eqref{eq:thm2pr2}, and \eqref{eq:thm2pr3}.
\end{proof}

\appendix

\section{Technical lemmas}{\label{appendix}}
\begin{proof}[Proof of Lemma \ref{le:A2}]
    (i) Since $\xi \sim \mathcal{N} (\mu, \eta \Sigma)$, straightforward calculations show that
    \begin{align*}
        &\E \left[ \exp(\abs{\xi}) \mathbf{1}_{\R^d \setminus B (\mu, 1 / 3)} (\xi) \right]\\
        =& \int_{\R^d \setminus B (\mu, 1 / 3)} (2 \pi \eta)^{-\frac{d}{2}} (\det \Sigma)^{-\frac{1}{2}} \exp\left\{ \abs{x} - \frac{1}{2 \eta} \abs{ \Sigma^{-1 / 2} (x - \mu)}^2\right\} \dif x \\
        \leq& (2 \pi)^{-\frac{d}{2}} \eup^{\abs{\mu}} \int_{\R^d \setminus B (\mathbf{0}, 1 / (3\sqrt \eta \opnorm{\Sigma}^{1 / 2}))} \exp\left\{ \opnorm{\Sigma}^{1 / 2}\sqrt  \eta \abs{y} - \frac{1}{2} \abs{y}^2\right\} \dif y,\\
        =& (2 \pi)^{-\frac{d}{2}} \eup^{\abs{\mu}} \int_{\R^d \setminus B (\mathbf{0}, 1 / (3\sqrt \eta \opnorm{\Sigma}^{1 / 2}))} \exp\left\{- \frac{1}{4} (\abs{y}-2\opnorm{\Sigma}^{1 / 2}\sqrt \eta)^2+ \eta \opnorm{\Sigma}- \frac{1}{4}\abs{y}^2\right\} \dif y\\
        \leq& (2 \pi )^{-\frac{d}{2}} \eup^{\abs{\mu}+\eta \opnorm{\Sigma}} \int_{\R^d \setminus B (\mathbf{0}, 1 / (3\sqrt \eta \opnorm{\Sigma}^{1 / 2}))} \exp\left\{ -\frac14|y|^2\right\} \dif y\\
        \leq& C \exp\left\{ \abs{\mu}+\eta \opnorm{\Sigma}-C/(\eta \opnorm{\Sigma}) \right\}, 
    \end{align*}
    where in the first inequality we use the variable substitution $x = \sqrt{\eta}(\Sigma^{1/2} y + \mu)$ and the last inequality we use the formula of the tail probability of Gaussian distributions. 
    
    Since $\eta \opnorm{\Sigma} \leq 1 / 6$ and $\eup^{-C/(\eta \opnorm{\Sigma})}\leq C\eta$, we can get 
    \begin{align*}
        \E \left[ \eup^{\abs{\xi}} \mathbf{1}_{\R^d \setminus B (\mu, 1 / 3)} (\xi) \right] \leq C \eta \eup^{\abs{\mu}}.
    \end{align*} 

 (ii) Notice that
    \begin{align} \label{eq:ximu}
    \begin{split}
        \abs{\xi}^2 \abs{\mu}^2 - \sca{\xi, \mu}^2
        &= (\abs{\xi - \mu}^2 + 2 \sca{\xi - \mu, \mu} + \abs{\mu}^2) \abs{\mu}^2 - (\sca{\xi - \mu, \mu} + \abs{\mu}^2)^2\\
        &\leq \abs{\xi - \mu}^2 \abs{\mu}^2.
        \end{split}
    \end{align}
    
    Combining $\abs{\mu} \geq 2 / 3$ and $\abs{\xi - \mu} < 1 / 3$ derives that $\abs{\xi} \geq \abs{\mu}-\abs{\xi - \mu}\geq 1 / 3$ and
    \begin{align*}
        \sca{\xi, \mu}
        = \abs{\mu}^2 + \sca{\xi - \mu, \mu}
        \geq (\abs{\mu} - \abs{\xi - \mu}) \abs{\mu}
        \geq (1 / 3){\abs{\mu}} ,
    \end{align*}
    which implies that 
    \begin{align}\label{eq:ximu1}
    \abs{\xi} \abs{\mu} + \sca{\xi, \mu} \geq (2 / 3) \abs{\mu}.
    \end{align}
    
    By \eqref{eq:ximu} and \eqref{eq:ximu1}, we have
    \begin{align*}
        \abs{\xi} \abs{\mu} - \sca{\xi, \mu}
        = \frac{\abs{\xi}^2 \abs{\mu}^2 - \sca{\xi, \mu}^2}{\abs{\xi} \abs{\mu} + \sca{\xi, \mu}}
        \leq (3 /2)\abs{\xi - \mu}^2 \abs{\mu},
    \end{align*}
    which implies $\abs{\xi} \leq \sca{\xi, \mu} / \abs{\mu} + (3 / 2) \abs{\xi - \mu}^2$. 
    
    It follows that
    \begin{align} \label{eq:ximu2}
    \begin{split}
        &\E \left[ \eup^{\abs{\xi}} \mathbf{1}_{B (\mu, 1 / 3)} (\xi) \right]\\
        \leq& \E \exp \left\{ \langle \xi, \frac{\mu}{\abs{\mu}} \rangle + \frac{3}{2} \abs{\xi - \mu}^2\right\} \\
        = &\int_{\R^d} (2 \pi \eta)^{-\frac{d}{2}} (\det \Sigma)^{-\frac{1}{2}} \exp \left\{ \frac{\sca{x, \mu}}{\abs{\mu}} + \frac{3}{2} \abs{x - \mu}^2 - \frac{1}{2 \eta} \abs{\Sigma^{-\frac12} (x - \mu)}^2 \right\} \dif x \\
        \leq& \int_{\R^d} (2 \pi \eta)^{-\frac{d}{2}} \exp \left\{ - \frac{1 - 3 \opnorm{\Sigma} \eta}{2 \eta} \abs{y}^2 + \frac{\sca{y, \Sigma^{\frac12} \mu}}{\abs{\mu}} + \abs{\mu} \right\} \dif y \\
        =& (1 - 3 \opnorm{\Sigma} \eta)^{-\frac{d}{2}} \exp \left\{ \frac{\abs{\Sigma^{\frac12} \mu}^2}{2 (1 - 3 \opnorm{\Sigma} \eta) \abs{\mu}^2} \eta + \abs{\mu} \right\},
    \end{split}
    \end{align}
    where in the second inequality we use the variable substitution $x = \Sigma^{\frac12} y + \mu$. 
    
    By the fact that $\eta \opnorm{\Sigma} \leq 1 / 6$, we have 
    \begin{align*}
    1 - 3 \opnorm{\Sigma} \eta \geq 1 / 2, \text{ and } 1 - 3 \opnorm{\Sigma} \eta \geq \exp \left\{- 6 \opnorm{\Sigma} \eta \right\},
    \end{align*} 
    so combining \eqref{eq:ximu2}, we can get
    \begin{align*}
        \E \left[ \eup^{\abs{\xi}} \mathbf{1}_{B (\mu, 1 / 3)} (\xi) \right]
        \leq \exp \left\{ \abs{\mu} + (3 d + 1) \opnorm{\Sigma}\eta\right\} \leq \eup^{\abs{\mu} + C \eta}. 
    \end{align*}
So the desired result follows.
\end{proof}

\begin{proof}[Proof of Lemma \ref{ergodic}]
    (i) Since $f \in \mathcal{B}_{b} (\R^d)$, it is clear that $\|P_{t}f\|_{\infty}\le \|f\|_{\infty}$. By the fact that any $f \in \mathcal{B}_{b} (\R^d)$ can be approximated almost everywhere by a sequence of $f_i \in \mathcal{C}_{b}^{1}(\R^d)$ satisfying $\norm{f_i}_\infty \leq 2 \norm{f}_\infty$ (see, for instance \cite[Theorem 7.10, 8.14]{Folland1999}), it suffices to show that for any $t>0$ there exists a constant $C_{t}$ such that 
    \begin{align*}
        \abs{\nabla_{v} P_{t}f(x)}\le C_{t}\norm{f}_{\infty}|v|V(x),\quad \forall x,v\in \R^{d},\quad  \forall f\in \mathcal{C}_{b}^{1}(\R^d).
    \end{align*}
    As a consequence of Lemma \ref{le:BEL}, \ref{le:BEL1} and Assumption \ref{A2}, we have
    \begin{align*}
        \abs{\nabla_{v}P_{t}f(x)}
        \le \frac{1}{t} \|f\|_{\infty} \sqrt{\int_{0}^{t}\E\abs{\sigma^{-1} (X_s^x) R_s^v}^{2} \dif s }
        \le \frac{1}{\sqrt{t}}e^{Ct}\|f\|_{\infty}|v|V(x),
    \end{align*}
    which implies the continuity of $P_{t}f$.

    \noindent (ii) By the definition of the irreducibility, it suffices to show that for any $x,y\in \R^{d}$ and $T>0$, $\delta>0$, 
    \begin{align*}
        \mathbb{P}\left(|X_{T}^{x}-y|<\delta  \right)>0.
    \end{align*}
    For any fixed $\epsilon>0$ and $t_{0}\in (0,T)$, set
    \begin{align} \label{eq:defeX}
        X^{\epsilon,x}_{t_{0}}:=X_{t_{0}}^{x} \mathbf{1}_{\left\{\abs{X_{t_{0}}^{x}}\le \epsilon^{-1}\right\}}.
    \end{align}
    Since Lemma \ref{le:Xmoment} shows that $\E|X_{t_{0}}^{x}|^{2}\le e^{-\lambda t_{0}}|x|^{2}+C$, it follows from dominated convergence theorem that
    \begin{align*}
        \lim_{\epsilon \downarrow 0}\E \abs{X_{t_{0}}^{x}-X^{\epsilon,x}_{t_{0}}}^{2}=\lim_{\epsilon \downarrow 0}\E \left[\abs{X_{t_{0}}^{x}}^{2} \mathbf{1}_{\left\{ \abs{X_{t_{0}}^{x}}>\epsilon^{-1} \right\}}\right]=0.
    \end{align*}
    For $t\in [t_{0},T]$, further denote
    \begin{align} \label{eq:defbarX}
        \Bar{X}_{t}^{\epsilon,x}:=\frac{T-t}{T-t_{0}}X^{\epsilon,x}_{t_{0}} +\frac{t-t_{0}}{T-t_{0}}y,\quad \text{and} \quad 
        \Bar{b}^{\epsilon}_{t}:=\frac{y-X^{\epsilon,x}_{t_{0}}}{T-t_{0}}-b(\Bar{X}_{t}^{\epsilon,x}).
    \end{align}
    It can be easily verified that
    \begin{align*}
        \Bar{X}_{t_{0}}^{\epsilon,x}=X^{\epsilon,x}_{t_{0}}, \quad  \Bar{X}_{T}^{\epsilon,x}=y,
    \end{align*}
    and
    \begin{align*}
        \Bar{X}_{t}^{\epsilon,x} = X^{\epsilon,x}_{t_{0}}+\int_{t_{0}}^{t} \left( b(\Bar{X}_{s}^{\epsilon,x})+\Bar{b}^{\epsilon}_{s} \right) \dif s.
    \end{align*}
    Now, consider the following SDE on $[0,T]$,
    \begin{align} \label{eq:defbarY}
    \begin{split}
        \Bar{Y}_{t}^{\epsilon,x}
        &:=x+\int_{0}^{t} \left( b(\Bar{Y}_{s}^{\epsilon,x})+\Bar{b}^{\epsilon}_{s} \mathbf{1}_{\{s>t_{0}\}} \right) \dif s+ \int_{0}^{t} \sigma(\Bar{Y}_{s}^{\epsilon,x}) \dif B_{s}\\
        &=x+\int_{0}^{t} b(\Bar{Y}_{s}^{\epsilon,x}) \dif s+ \int_{0}^{t} \sigma(\Bar{Y}_{s}^{\epsilon,x}) \dif \Tilde{B}_{s},
    \end{split}
    \end{align}
    where 
    \begin{align*}
        \Tilde{B}^{\epsilon}_{t}:=B_{t}+\int_{0}^{t} \sigma^{-1} (\Bar{Y}_{s}^{\epsilon,x}) \Bar{b}^{\epsilon}_{s} \mathbf{1}_{\{s>t_{0}\}} \dif s.
    \end{align*}
    By \eqref{eq:defeX}, \eqref{eq:defbarX} and Assumption \ref{A1}, \ref{A2}, $\abs{\sigma^{-1} (\Bar{Y}_{s}^{\epsilon,x}) \Bar{b}^{\epsilon}_{s}} \le C_{\epsilon,t_{0}}$, $\forall s \in (t_0, T)$ holds for some constant $C_{\epsilon,t_{0}}$ depending on $\epsilon$ and $t_{0}$. Hence, 
    \begin{align*}
        R^{\epsilon} := \exp \left\{ \int_{0}^{T} \left\langle \sigma^{-1} (\Bar{Y}_{t}^{\epsilon,x}) \Bar{b}^{\epsilon}_{t} \mathbf{1}_{\{t>t_{0}\}}, \dif B_{s} \right\rangle -\frac{1}{2}\int_{0}^{T}  \abs{\sigma^{-1} (\Bar{Y}_{t}^{\epsilon,x}) \Bar{b}^{\epsilon}_{t} \mathbf{1}_{\{t>t_{0}\}}}^{2}\dif s \right\},
    \end{align*}
    is a martingale and $\E R^{\epsilon}=1$. It then follows from the Girsanov's theorem that
    $(\Tilde{B}^{\epsilon}_{t})_{t\in [0,T]}$ is a Brownian motion under the probability measure $R^{\epsilon}\dif \PP$ with $\PP$ denoting the probability measure corresponding to $(B_{t})_{t\in [0,T]}$. Hence, $\Bar{Y}_{t}^{\epsilon,x}$ has the same law as $X_{t}^{x}$ under $R^{\epsilon}\dif \PP$ and to prove the desired result, it suffices to show that there exist a $t_{0}$ such that
    \begin{align*}
        \mathbb{P}\left( \abs{\Bar{Y}^{\epsilon,x}_{T}-y} <\delta \right)>0.
    \end{align*}

    According to Assumption \ref{A1} and Young's inequality,
    \begin{align*}
        \sca{y - x, b (y) - b (x)}
        \leq C (1 + \abs{x}^{r + 1} \abs{y} + \abs{x} \abs{y}^{r + 1}) - \lambda (\abs{x}^{r + 2} + \abs{y}^{r + 2})
        \leq C (1 + \abs{x}^{r + 2}).
    \end{align*}
    Together with It\^o's formula and Assumption \ref{A2}, we have
    \begin{align*}
        \frac{\dif}{\dif t}\E\abs{\Bar{Y}^{\epsilon,x}_{t}-\Bar{X}^{\epsilon,x}_{t}}^{2}
        &=2\E\left\langle \Bar{Y}^{\epsilon,x}_{t}-\Bar{X}^{\epsilon,x}_{t}, b(\Bar{Y}^{\epsilon,x}_{t})-b(\Bar{X}^{\epsilon,x}_{t})  \right\rangle+ \E\|\sigma(\Bar{Y}^{\epsilon,x}_{t})\|^{2}_{\mathrm{HS}} \\
        &\le C \left( 1 + \E\abs{\Bar{X}^{\epsilon,x}_{t}}^{r+2} \right).
    \end{align*}
    It follows from \eqref{eq:defbarX} and Lemma \ref{le:Xmoment} that $\E \lvert \Bar{X}^{\epsilon,x}_{t} \rvert^{r + 2} \le \E [ ( \lvert X_{t_{0}}^{x} \rvert + \abs{y} )^{r + 2} ] \le C$, which implies
    \begin{align*}
        \E\abs{\Bar{Y}^{\epsilon,x}_{T}-\Bar{X}^{\epsilon,x}_{T}}^{2}
        \le  \E\abs{\Bar{Y}^{\epsilon,x}_{t_{0}}-\Bar{X}^{\epsilon,x}_{t_{0}}}^{2}+C(T-t_{0})
        =\E\abs{X^{x}_{t_{0}}-X^{\epsilon,x}_{t_{0}}}^{2}+C(T-t_{0}).
    \end{align*}
    Hence
    \begin{align*}
        \mathbb{P}\left(\abs{\Bar{Y}^{\epsilon,x}_{T}-y} \geq \delta  \right)
        =\mathbb{P}\left(\abs{\Bar{Y}^{\epsilon,x}_{T}-\Bar{X}^{\epsilon,x}_{T}} \geq \delta  \right)
        \le \frac{ \E\abs{X^{x}_{t_{0}}-X^{\epsilon,x}_{t_{0}}}^{2}+C(T-t_{0})}{\delta^{2}},
    \end{align*}
    where the constant $C$ does not depend on $\epsilon$ and $t_0$. Choosing $t_{0}$ sufficiently close to $T$ and $\epsilon$ sufficiently small yields that
    \begin{align*}
        \mathbb{P} \left( \abs{\Bar{Y}^{\epsilon,x}_{T}-y} \geq \delta \right) < 1.
    \end{align*}
    So the desired result follows.
\end{proof}

\begin{proof}[Proof of Lemma \ref{le:A1}]
(i) 
By simple calculation, we can obtain
\begin{align*}
    \eta_k^{1 + \beta} \eup^{-c (t_n - t_k)}
    &\leq  \eta_k^\beta \eup^{-c (t_n - t_{k})} ((\eup^{c \eta_k} - 1)/{c}) \\
    &\leq \frac{\eup^c}{c} \eta_k^\beta \eup^{-c (t_n - t_{k-1})} (\eup^{c \eta_k} - 1)\\
    &= \frac{\eup^c}{c} \eta_k^\beta \left[ \eup^{-c (t_n - t_k)} - \eup^{-c (t_n - t_{k-1})} \right],
\end{align*}
where the first inequality comes from $\eta_k \leq (\eup^{c \eta_k} - 1) / c$, and the second inequality comes from $\eup^{-c (t_n - t_k)} \leq \eup^{c - c (t_n - t_{k-1})}$.

Since $\eta_{k-1}^\beta - \eta_k^\beta\leq \beta \eta_k^{\beta - 1} (\eta_{k-1} - \eta_k)\leq \beta \theta \eta_k^{1 + \beta}$ by Assumption \ref{A3}, we have
\begin{align*}
    &\mathrel{\phantom{=}} \sum_{k = 1}^n \eta_k^{1 + \beta} \eup^{-c (t_n - t_k)} 
    \leq \frac{\eup^c}{c} \sum_{k = 1}^n \eta_k^\beta \left[ \eup^{-c (t_n - t_k)} - \eup^{-c (t_n - t_{k-1})} \right] \\
    &\qquad\qquad= \frac{\eup^c}{c} \left[\sum_{k = 1}^n \left( \eta_k^\beta \eup^{-c (t_n - t_k)} - \eta_{k-1}^\beta \eup^{-c (t_n - t_{k-1})} \right) +  \sum_{k = 1}^n \left( \eta_{k-1}^\beta - \eta_k^\beta \right) \eup^{-c (t_n - t_{k-1})}\right] \\
    &\qquad\qquad\leq \frac{\eup^c}{c} \left(\eta_n^\beta +  \beta \theta    \sum_{k = 1}^n \eta_k^{1 + \beta} \eup^{-c (t_n - t_k)}\right).
\end{align*}
Then, it follows from $\theta < c \eup^{-c} / \beta$ that
\begin{align*}
    \sum_{k = 1}^n \eta_k^{1 + \beta} \eup^{-c (t_n - t_k)}
    \leq \frac{\eta_n^\beta}{c \eup^{-c} - \beta \theta}\leq C \eta_n^\beta.
\end{align*}

(ii) By $\eta_{k-1} \leq \eta_k (1 + \theta \eta_k)$ from Assumption \ref{A3}, we know that for $K_n \leq m \leq n-1$,
\begin{align*}
    \frac{\eta_{m}}{\eta_n}
    = \prod_{k = m + 1}^n \frac{\eta_{k-1}}{\eta_k}
    \leq \prod_{k = m + 1}^n (1 + \theta \eta_k)
    \leq \prod_{k = m + 1}^n \eup^{\theta \eta_k}
    = \eup^{\theta (t_n - t_{m})}
    \leq \eup^\theta.
\end{align*}
Combining the fact that $t_n - t_k \geq (n-k) \eta_n$, we have
\begin{align*}
    \sum_{k = K_n}^{n-1} \frac{\eta_k^{1 + \beta}}{\sqrt{t_n - t_k}}
    &\leq \sum_{k = K_n}^{n-1} \frac{\eup^{\theta (1 + \beta)} \eta_n^{\frac{1}{2} + \beta}}{\sqrt{n - k}} \\
    &\leq \eup^{\theta (1 + \beta)} \eta_n^{\frac{1}{2} + \beta} \int_0^{n - K_n} \frac{1}{\sqrt{x}} \, \dif x \\
    &= 2 \eup^{\theta (1 + \beta)} \eta_n^\beta \sqrt{(n - K_n) \eta_n} \\
    &\leq 2 \eup^{\theta (1 + \beta)} \eta_n^\beta\leq C \eta_n^\beta,
\intertext{and}
    \sum_{k = K_n}^{n-1} \frac{\eta_k^{1 + \beta}}{t_n - t_k}
    &\leq \sum_{k = K_n}^{n-1} \frac{\eup^{\theta (1 + \beta)} \eta_n^\beta}{n - k} \\
    &\leq \eup^{\theta (1 + \beta)} \eta_n^\beta \left( 1 + \int_1^{n - K_n} \frac{1}{x} \, \dif x \right) \\
    &= \eup^{\theta (1 + \beta)} \eta_n^\beta \left\{ 1 + \ln [(n - K_n) \eta_n] - \ln \eta_n \right\} \\
    &\leq \eup^{\theta (1 + \beta)} \left( \frac{1}{\abs{\ln \eta_1}} + 1 \right) \eta_n^\beta \abs{\ln \eta_n}\leq C\eta_n^\beta \abs{\ln \eta_n}.  
\end{align*}
So the desired result follows.
\end{proof}
\bibliographystyle{amsplain}
\bibliography{Tamed_Euler-Maruyama_Method_for_SDEs_with_Non-globally_Lipschitz_Drift_and_multiplicative_Noise}

\end{document}